\documentclass[11pt]{amsart}

\calclayout

\usepackage{color}
\usepackage{verbatim}
\usepackage[noadjust]{cite}

\usepackage[utf8x]{inputenc}
\usepackage{t1enc}
\usepackage[english]{babel}
\usepackage{pgfplots}
\pgfplotsset{compat=1.15}
\usepackage{mathrsfs}
\usetikzlibrary{arrows}

\definecolor{ffqqqq}{rgb}{1.,0.,0.}
\definecolor{uuuuuu}{rgb}{0.26666666666666666,0.26666666666666666,0.26666666666666666}

\usepackage{amsmath,amssymb,amsthm}
\usepackage{mathrsfs,esint}
\usepackage{graphicx,wrapfig}
\usepackage[format=hang]{caption}
\newcommand{\R}{\mathbb{R}}

\newcommand{\N}{\mathbb{N}}
\newcommand{\Z}{\mathbb{Z}}

\newcommand{\A}{\mathcal{A}}

\newcommand{\Ha}{\mathcal{H}}

\newcommand{\eps}{\varepsilon}

\newcommand{\loc}{\text{loc}}
\newcommand{\Om}{\Omega}

\newcommand{\bmat}{\begin{bmatrix}}
\newcommand{\emat}{\end{bmatrix}}

\newcommand{\wtil}{\widetilde}

\newcommand{\oXeps}{\overline X_\eps}
\newcommand{\vcapXp}{\operatorname{cap}_p^{\oXeps}}
\newcommand{\vcap}{\operatorname{cap}}

\providecommand*{\vint}[1]{\mathchoice
          {\mathop{\vrule width 5pt height 3 pt depth -2.5pt
                  \kern -9pt \kern 1pt\intop}\nolimits_{\kern -5pt{#1}}}
          {\mathop{\vrule width 5pt height 3 pt depth -2.6pt
                  \kern -6pt \intop}\nolimits_{\kern -3pt{#1}}}
          {\mathop{\vrule width 5pt height 3 pt depth -2.6pt
                  \kern -6pt \intop}\nolimits_{\kern -3pt{#1}}}
          {\mathop{\vrule width 5pt height 3 pt depth -2.6pt
                  \kern -6pt \intop}\nolimits_{\kern -3pt{#1}}}}
\newcommand*{\jint}{\fint}
\DeclareMathOperator{\Mod}{Mod}

\DeclareMathOperator{\dist}{dist}
\DeclareMathOperator{\BCap}{Cap}
\DeclareMathOperator{\diam}{diam}
\DeclareMathOperator{\rad}{rad}
\DeclareMathOperator{\supt}{supt}

\DeclareMathOperator{\Id}{Id}

\numberwithin{equation}{section}
\theoremstyle{plain}
\newtheorem{thm}[equation]{Theorem}

\newtheorem{prop}[equation]{Proposition}

\newtheorem{lem}[equation]{Lemma}

\theoremstyle{definition}

\newtheorem{defn}[equation]{Definition}

\newtheorem{remark}[equation]{Remark}

\newtheorem{exa}[equation]{Example}

\makeatletter
\def\blfootnote{\xdef\@thefnmark{}\@footnotetext}
\makeatother

\begin{document}

\title[Well-posedness for inhomogeneous Dirichlet problems]{Well-posedness of Dirichlet boundary value problems for reflected fractional $p$-Laplace-type inhomogeneous equations in compact doubling metric measure spaces}


\author{Josh Kline}
\address{Department of Mathematical Sciences, University of Cincinnati, P.O. Box 210025, Cincinnati, OH 45221-0025, USA.}
\email{klinejp@ucmail.uc.edu}
\author{Feng Li}
\address{Department of Mathematical Sciences, University of Cincinnati, P.O. Box 210025, Cincinnati, OH 45221-0025, USA.}
\email{lif9@ucmail.uc.edu}
\author{Nageswari Shanmugalingam}
\address{Department of Mathematical Sciences, University of Cincinnati, P.O. Box 210025, Cincinnati, OH 45221-0025, USA.}
\email{shanmun@uc.edu} 
\thanks{\noindent {\bf Acknowledgment:} J.K.'s research was partially funded by grant DMS~\#2054960 and a Taft Dissertation 
	Fellowship from the Taft Research Center.
	N.S.'s research was partially supported by grant DMS~\#2054960.}
	

\date{\today}

\maketitle

\begin{abstract}
In this paper we consider the setting of a locally compact, non-complete metric measure space $(Z,d,\nu)$
equipped with a doubling measure $\nu$,
under the condition that the boundary $\partial Z:=\overline{Z}\setminus Z$
(obtained by considering the completion of $Z$) supports a Radon measure $\pi$
which is in a $\sigma$-codimensional relationship to $\nu$ for some $\sigma>0$. 
We explore existence, uniqueness, comparison property, and stability properties of solutions to 
inhomogeneous Dirichlet problems associated with certain nonlinear nonlocal operators on $Z$. 
We also establish interior regularity of solutions when the inhomogeneity data is in an
$L^q$-class for sufficiently large $q>1$, and verify a Kellogg-type property when the inhomogeneity data vanishes
and the Dirichlet data is continuous.
\end{abstract}

\section{Introduction}

	In the present paper, we are interested in a nonlocal Dirichlet problem, namely the one given in 
	Definition~\ref{def:Dirichlet1} in the nonsmooth setting of compact metric measure space equipped with a doubling measure.
	The Dirichlet condition is imposed in the sense of traces~\eqref{eq:trace-defn}.
	The operator in question is the trace of a $p$-Laplace-type operator on a ``higher dimensional'' uniform domain that
	sees the compact metric space as its boundary. This is part of an overarching program of
	exploring properties of the Dirichlet-to-Neumann operator on the compact space, first initiated in the
	setting of complete doubling metric measure spaces in~\cite{EGKSS, CKKSS, CGKS}.
	In the classical Euclidean setting, with $Z$ a bounded smooth domain in $\R^n$ and $p=2$, 
	the problem stated in Definition~\ref{def:Dirichlet1} corresponds to the problem
	\[
	\begin{cases}
		(-\Delta_\Om)^\theta u=G&\text{ in }\Omega\\
		u=f&\text{ on }\partial \Omega.
	\end{cases}
	\]
	Here, $\Delta_\Om$ is the uniformly elliptic divergence-form operator on $\Om$ given by
	$\Delta_\Om u(x)=-\text{div}(A(x)\nabla u(x))$ on the Euclidean domain $\Om$, see for example~\cite{CSt}.
	We keep in mind that there are other approaches to setting up fractional/nonlocal problems in the Euclidean
	setting as well, as for example found in the works of Warma~\cite{Warma}, which were based
	directly on the Besov energy, in contrast to the approach of~\cite{CSt}, which is based more on
	the trace of a Laplace-Beltrami operator on a higher dimensional space. Unless the domain in question is the
	entire Euclidean space, it is not known that the operator studied by Caffarelli and Stinga~\cite{CSt} is the 
	same one studied in~\cite{Warma}. When the Euclidean domain is the entire Euclidean space and $p=2$,
	the approach of constructing the fractional operator via such an extension agrees with the one studied in~\cite{Warma}
	and with the fractional power of the classical Laplacian on the Euclidean space.  For more studies involving various aspects of analogs of
	this problem in the Euclidean setting, see \cite{CS, CSt, ST, KLL, BII, Rossi, Warma, GuanMa, JaBjo}, for a sampling, and as well as \cite{FF} in the Carnot group setting. In the more general metric setting,
	if the geometric properties of the metric space allows for a doubling measure that supports a $2$-Poincar\'e inequality,
	then the three approaches of constructing fractional operators yields the same operator as well, see for instance~\cite{EGKSS};
	however, when the space is equipped with a doubling measure but does not support a $2$-Poincar\'e inequality,
	there may not be a natural choice of a Laplacian as a local operator; in this case, the approaches of Warma~\cite{Warma}
	and~\cite{CSt} may differ; the present paper focuses on the approach of~\cite{CSt} due to its links to the
	Dirichlet-to-Neumann operators. Moreover, we consider non-linear analogs of fractional powers $p$-Laplace operator
	even in the setting where there is no natural (local) $p$-Laplace operator. Examples such as the vonKoch snowflake curve,
	the Sierpinski carpet, and more general fractal objects do not have a natural $p$-Laplace operator, but when
	equipped with their natural Hausdorff measures, are doubling spaces in which the results of the present paper apply
	but no other approaches are currently available, see Section~\ref{sec:Examples}.

	Aspects of the above problem have also been studied in the setting of Dirichlet forms in~\cite{BGMN,Chen} (for the case $p=2$).
This interpretation is known in literature on probability as the 
	problem corresponding to the killed process.
	In considering the killed process, the problem can be re-stated as an energy minimization problem, but the associated
	the energy  takes into account the oscillation of the functions outside of the domain as well, subject to the constraint that
	the functions agree with $f$ outside the domain. 
	More particularly, given a function $u$ on $\R^n$, the energy $\mathcal{E}_G(u)$ is a globally defined \emph{non-local} energy
	whose value is changed even if $u$ is only changed outside of the domain $\Om$. One then attempts
	to find a function $u$ with minimal energy among all functions $v$ with $v=f$ in $\R^n\setminus\Om$.
	Thus, if we perturb the value of $f$ far away from $\Om$, the 
	solution to the above minimization problem will change even though the boundary values of $f$ may not change.
	This problem is well-understood in the Euclidean setting, and the corresponding function-space where the
	energy has as a component 
	the globally-defined Besov energy associated with the Besov
	space $B^\theta_{p,p}(\R^n)$; these are functions $u\in L^p_{loc}(\R^n)$ for which
	\[
	\int_{\R^n}\int_{\R^n}\frac{|u(y)-u(x)|^p}{|y-x|^{n+\theta p}}\, dy\, dx<\infty.
	\]
	We refer the interested reader to~\cite{Chen, KLL,BII, Rossi, JaBjo, GuanMa} and 
	the references therein for more on this topic.
	
	In contrast, in this paper we are interested in trace processes that are related to the so-called reflected jump processes, or censored processes, see \cite{BBC}.  In such situations, competitor functions need only address the boundary data in the sense of traces, rather than agreeing with prescribed data in the entirety of the exterior of the domain. Thus, if we perturb the value of $f$ far away from $\Om$, the 
	solution to the above problem will not change because the boundary values of $f$ does not change.
	This is a more nuanced problem, for one also then needs to make sense of traces of Besov functions, given that
	Besov functions can be arbitrarily perturbed on sets of measure zero.
	Due to this difficulty, this problem does not have as extensive a literature presently, in comparison to the problem related
	to the killed process. In the Euclidean setting, this problem has been studied in~\cite{Warma},  and~\cite{Warma} 
has an excellent exposition on the background related to the problem in the Euclidean setting, with
substantial literature in the bibliography.

Our precise setting is the following: we consider a a locally compact, non-complete bounded metric measure space $(Z,d,\nu)$, with $\nu$ a doubling measure, whose boundary is $\partial Z=\overline Z\setminus Z$, where $\overline Z$ is the completion of $Z$.  We extend $\nu$ to $\partial Z$ by the zero extension, so that $\nu(\partial Z)=0$.   For $\sigma>0$, we then equip $\partial Z$ with the codimension $\sigma$ Hausdorff measure $\pi:=\Ha^{-\sigma}_\nu$, see \eqref{eq:def-pi}, and assume that $0<\pi(\partial Z)<\infty$.  In order to rigorously frame our problem given by Definition~\ref{def:Dirichlet1} in this setting, we first establish the following trace theorem pertaining to Besov spaces on $Z$. 

\begin{thm}\label{thm:TraceTheorem}
Let $0<\theta<1$ and $1<p<\infty$, and let $u\in B^\theta_{p,p}(\overline Z,\nu)$.
\begin{itemize}
\item[(i)]If $p>\max\{1,\sigma/\theta\}$ and if there is a sequence $u_k\in B^\theta_{p,p}(\overline Z,\nu)$, compactly supported in $Z$, such that $u_k\to u$ in $B^\theta_{p,p}(\overline Z,\nu)$, then for $\pi$-a.e.\ $\zeta\in\partial Z$, 
\begin{equation}\label{eq:IntroTrace}
\lim_{r\to 0^+}\fint_{B(\zeta,r)}|u|^pd\nu=0.
\end{equation}  
\item[(ii)] If \eqref{eq:IntroTrace} holds for $\BCap_{\theta,p}^{\overline Z}$-a.e.\ $\zeta\in\partial\Omega$, then there is a sequence $u_k\in B^\theta_{p,p}(\overline Z,\nu)$, compactly supported in $Z$, such that $u_k\to u$ in $B^\theta_{p,p}(\overline Z,\nu)$. 
\end{itemize}
Furthermore, if $p>\max\{1,\sigma/\theta\}$ and $\pi$ is $\sigma$-codimensional Ahlfors regular with respect to $\nu$ (see Definition~\ref{def:codim-regular}), then there exist bounded  linear trace and extension operators
\[
T_Z:B^\theta_{p,p}(\overline Z,\nu)\to B^{\theta-\sigma/p}_{p,p}(\partial Z,\pi),\qquad E_Z:B^{\theta-\sigma/p}_{p,p}(\partial Z,\pi)\to B^\theta_{p,p}(\overline Z,\nu)
\]
such that $T_Z\circ E_Z$ is the identity map on $B^{\theta-\sigma/p}_{p,p}(\partial Z,\pi)$.  For each $u\in B^\theta_{p,p}(\overline Z,\nu)$, the trace operator also satisfies
\[
\lim_{r\to 0^+}\fint_{B(\zeta,r)}|u-T_Zu(\zeta)|d\nu=0
\]
for $\pi$-a.e.\ $\zeta\in\partial Z$.
\end{thm}
	
Theorem~\ref{thm:TraceTheorem} is proven via Proposition~\ref{prop:1}, Proposition~\ref{prop:2}, and Lemma~\ref{lem:traceZ}, using trace and extension results from \cite{BBS} between Besov spaces on $Z$ and Newton-Sobolev spaces on hyperbolic fillings of $Z$.

Having established the relevant trace theory, we then study well-posedness for the inhomogeneous Dirichlet problem given by Definition~\ref{def:Dirichlet1}.  The nonlocal (and nonlinear) operator associated with the problem is given via an energy form $\mathcal{E}_T$ as in Definition~\ref{def:ET}.  This operator was introduced in \cite{CKKSS}, where the approach of \cite{CS, CSt} towards constructing nonlocal operators via Dirichlet-to-Neumann maps was adapted to our current metric setting. The following theorem is a compilation of existence, uniqueness, comparison property,
and stability results established in this paper. 

\begin{thm}\label{thm:main}
Let $(Z,d)$ be a locally compact, non-complete metric space equipped with a doubling measure $\nu$, and $\sigma>0$ such that 
$\partial Z:=\overline{Z}\setminus Z$ is equipped with a Radon measure $\pi$ that is $\sigma$-codimensional with respect to $\nu$.
Let $\max\{1,\sigma\}<p<\infty$ and $\sigma/p<\theta<1$. Then the following claims hold true:
\begin{enumerate}
	\item For each $f\in B^{\theta-\sigma/p}_{p,p}(\partial Z,\pi)$ and each $G\in L^{p'}(Z,\nu)$ there exists exactly one
$u\in B^\theta_{p,p}(\overline Z,\nu)$ such that 
\begin{align*}
	\mathcal{E}_T(u,v)&=\int_Z G\, v\, d\nu\text{ whenever }v\in B^\theta_{p,p}(\overline Z,\nu)\text{ with }T_Zv=0\ \ \ \text{ $\pi$-a.e.~in }\partial Z,\\
	T_Zu&=f \ \ \ \text{ $\pi$-a.e. in }\partial Z.
\end{align*}
Such a solution $u$ is called a solution to the $(f,G)$-inhomogeneous Dirichlet problem.
    \item If $f_1,f_2\in  B^{\theta-\sigma/p}_{p,p}(\partial Z,\pi)$ with $f_1\le f_2$ $\pi$-a.e.~in $\partial Z$ and 
$G\in L^{p'}(Z,\nu)$, then with $u_i$ the solution to the $(f_i,G)$-inhomogeneous Dirichlet problem for $i=1,2$, we have that
$u_1\le u_2$ $\nu$-a.e.~in $Z$.
	\item If $(G_k)_k$ is a Cauchy sequence in $L^{p'}(Z,\nu)$ with limit $G$ and $(f_k)_k$ is a Cauchy sequence in 
$B^{\theta-\sigma/p}_{p,p}(\partial Z,\pi)$ with limit $f$, and if $u_k$ is the solution to the $(f_k,G_k)$-inhomogeneous Dirichlet problem,
then $u_k\to u$ in $B^\theta_{p,p}(\overline Z,\nu)$ with $u$ the solution to the $(f,G)$-inhomogeneous problem.
\end{enumerate}
\end{thm}

Theorem~\ref{thm:main} is proved in stages; the existence and uniqueness Claim~(1) is established as 
Theorem~\ref{thm:Inhom Existence}, 
comparison property Claim~(2) as
Proposition~\ref{prop:compare}, and the stability Claim~(3) as Theorem~\ref{thm:stable}.
The proof of existence of the solution is accomplished using the direct method of calculus of variation, but the
difficulty is in the details, as the problem is neither local nor is it homogeneous. To overcome this difficulty we
combined tools of potential theory related to Sobolev spaces with the direct method. The stability of the problem 
rests on two data here-the inhomogeneity data $G$ and the boundary data $f$. Given the nonlinearity and inhomogeneity, together
with the non-locality, of the problem, the standard way of proving stability for the $p$-Laplace equation does not work here,
and so we had to employ a two-step process in combination with the results from~\cite{BBS, CKKSS} in order to
establish the stability result. We point out here that if $\partial Z$ is empty, that is, there is no boundary data to consider,
the global results of~\cite{CKKSS} already yields the stability of solutions with respect to the inhomogeneity data $G$, but the
introduction of the non-empty boundary means that we cannot appeal to~\cite{CKKSS}.

In considering matters of continuity of solutions to the problem, we have the following result.

\begin{thm}\label{thm:main2}
	Under the setting of Theorem~\ref{thm:main}, let $f\in B^{\theta-\sigma/p}_{p,p}(\partial Z,\pi)$ and
$G\in L^{p'}(Z,\nu)$. Then with $Q_\nu$ the lower mass bound exponent associated with the measure $\nu$
in the sense of~\eqref{eq:LMB},
\begin{enumerate}
	\item if $G$ also belongs to $L^q(Z,\nu)$ for some $q>\max\{1,Q_\nu/(p\theta)\}$, then the solution $u$ to
	the $(f,G)$-inhomogeneous Dirichlet problem is locally H\"older continuous in $Z$.
	\item if $\partial Z$ is uniformly perfect and
	$f$ is continuous on $\partial Z$, then with $u$ the solution to the $(f,0)$-inhomogeneous Dirichlet problem, 
	we have that for every $\xi\in\partial Z$,
	\[
	   \lim_{Z\ni z\to \xi}u(z)=f(\xi).
	\]
\end{enumerate}
\end{thm}

Note that if $Q$ is the lower mass bound exponent associated with the auxiliary measure $\mu_\beta$ on the uniformized
hyperbolic filling $X_\eps$ of $\overline{Z}$ (as described in subsection~\ref{sec:HypFill}), then $Q_\nu=Q-\beta/\eps$. 

The interior regularity Claim~(1) of Theorem~\ref{thm:main2} is established in this paper via Theorem~\ref{thm:Holder}, and
the Kellogg property Claim~(2) of Theorem~\ref{thm:main2} is proved via Theorem~\ref{thm:HomogeneousKellogg}.
The interior regularity result follows immediately from~\cite{CGKS}, and Theorem~\ref{thm:Holder}
merely points this out. The Kellogg property would also follow immediately from~\cite{BjMacSh} if we know that
$\partial Z$ is uniformly $p$-fat (with respect to the potential theory related to the $p$-Laplacian on the filling $X_\eps$).
The uniform $p$-fatness is established in the proof of Theorem~\ref{thm:HomogeneousKellogg}, and then the
results of~\cite{BjMacSh} are invoked to complete the proof of the Kellogg property. We note here that we do not presently have
knowledge of Kellogg property if the inhomogeneity data $G$ does not vanish in a neighborhood of $\partial Z$,
and the classical proof for inhomogeneous but local $p$-Laplace equation does not apply here.

In summary, Theorem~\ref{thm:main} and Theorem~\ref{thm:main2} are established in Sections~5--7. Before doing so, we provide 
a description of the various notions associated with the nonsmooth metric setting in Section~2, and Sections~3 and~4 describe
traces of Sobolev and Besov functions at boundaries of domains.


\section{Preliminaries}

In this section we gather together the notions and tools used throughout the paper. The triple $(W,d_W,\mu_W)$ will stand in for a generic metric
measure space. For $w\in W$ and $r>0$ we denote the ball $\{z\in W\, :\, d_W(z,w)<r\}$ by $B_W(w,r)$. If the metric space $W$ is understood in context,
we can also drop the subscript and denote the ball by $B(w,r)$.

 Throughout this paper, we let $C$ denote a constant which depends only on the structural constants of the metric space (such as the doubling constant), unless otherwise specified.  Its precise value is not of interest to us and may change with each use, even within the same line.  Furthermore, given quantities $A$ and $B$, we will often use the notation $A\simeq B$ to mean that there exists a constant $C\ge 1$ such that $C^{-1}A\le B\le CA$.  Likewise, we use $A\lesssim B$ and $A\gtrsim B$ if the left and right inequalities hold, respectively.

\subsection{Metric and measure-theoretic notions}

\begin{defn}\label{def:unif-perf}
Given a metric space $(W,d_W)$ and a set $A\subset W$, we say that $A$ is \emph{uniformly perfect} if there is a constant $K\ge 2$ such that
whenever $w\in A$ and $0<r<\diam(A)$, the annulus $B(w,r)\setminus B(w,r/K)$ has a nonempty intersection with $A$. If $A=W$, we say that
the metric space $W$ is uniformly perfect.
\end{defn}

Given a set $A\subset W$, we
say that the measure $\mu_W$ is \emph{doubling on $A$}, if there exists $C_d\ge 1$ such that for all $w\in A$ and for all $r>0$, we have 
\begin{equation*}
0<\mu_W(B(w,2r))\le C_d\mu_W(B_W(w,r))<\infty.
\end{equation*}
When $A=W$, we merely say that $\mu_W$ is doubling.
If $\mu_W$ is doubling, then by iterating the above inequality, there exists $Q\ge 1$ and $C\ge 1$, depending on $C_d$, such that
\begin{equation}\label{eq:LMB}
\frac{\mu_W(B(w,r))}{\mu_W(B(z,R))}\ge C^{-1}\left(\frac{r}{R}\right)^Q
\end{equation}
for all $z\in W$, $0<r\le R$, and $w\in B(z,R)$. 

When $W$ is not complete, by $\partial W$ we mean the set $\overline{W}\setminus W$, with $\overline{W}$ the metric completion of $W$. 
We can extend
$\mu_W$ to $\partial W$ as a null measure.
We equip $\partial W$ with the \emph{codimension $\sigma$ Hausdorff measure} with respect to $\mu_W$, denoted by 
$\Ha_{\mu_W}^{-\sigma}$, with $\sigma>0$, as follows. For sets $A\subset\partial W$, we set
\begin{equation}\label{eq:def-pi}
\Ha_{\mu_W}^{-\sigma}(A)=\lim_{\eps\to 0^+}\Ha^{-\sigma}_{\mu_W,\eps}(A),
\end{equation}
where 
\[
\Ha^{-\sigma}_{\mu_W,\eps}(A):=\inf_{\{B_i\}_{i\in I\subset\N}} \sum_{i\in I}\frac{\mu_W(B_i)}{\rad(B_i)^\sigma}
\]
with the infimum over all covers $\{B_i\}_{i\in I\subset\N}$ of $A$ with balls $B_i\subset \overline{W}$ such that
$\rad(B_i)<\eps$ for each $i\in I$. We extend $\Ha_{\mu_W}^{-\sigma}$ to $W$ as a null measure. Thus we consider $\mu_W$ 
and $\Ha_{\mu_W}^{-\sigma}$ to
be Radon measures on $\overline{W}$, the first charging only subsets of $W$ and the second charging
only subsets of $\partial W$.

\begin{defn}\label{def:codim-regular}
Let $s>0$.
We say that a measure $\mu_0$ on $\overline W$ such that $\mu_0(W)=0$, 
is $s$-codimensional Ahlfors regular with respect to the measure $\mu_W$ if there is a constant $C\ge 1$ such that 
\begin{equation}\label{eq:codim-def}
C^{-1}\, \frac{\mu_W(B(\xi,r))}{r^s}\le \mu_0(B(\xi,r))\le C\, \frac{\mu_W(B(\xi,r))}{r^s}
\end{equation}
whenever $\xi\in\partial W$ and $0<r<2\, \diam(W)$.
\end{defn}

\begin{lem}
	Suppose that $\mu_W$ is doubling on $W$ (and hence on $\overline W$).
A measure $\mu_0$ on the space $\overline W$, with $\mu_0(W)=0$, is $\sigma$-codimensional with respect to $\mu_W$ if and only if 
$\Ha_{\mu_W}^{-\sigma}$ is $\sigma$-codimensional with respect to $\mu_W$ and $\mu_0\simeq \Ha_{\mu_W}^{-\sigma}$.
\end{lem}

\begin{proof}
If both $\Ha_{\mu_W}^{-\sigma}$ and $\mu_0$ are $\sigma$-codimensional with respect to $\mu_W$, then clearly $\mu_0\approx\Ha_{\mu_W}^{-\sigma}$; 
so it
suffices to show that if there is a measure $\mu_0$ on $\overline W$ that is $\sigma$-codimensional with respect to $\mu_W$,
then so is $\Ha_{\mu_W}^{-\sigma}$. To this end, we assume that $\mu_0$ is $\sigma$-codimensional with respect to $\mu_W$.
As $\mu_W$ is doubling, it follows that $\mu_0$ is doubling on $\partial W$.

Let $\xi\in\partial W$. Fix $\eps>0$ and
let $\{B_i\}_{i\in I\subset\N}$ be a cover of $B(\xi,r)$ in $\overline{W}$ with $\rad(B_i)<\eps$ for 
each $i\in I$. Let $I_1$ be the collection of all $i\in I$ for which $B_i\cap\partial W$ is non-empty, and for
each $i\in I_1$ we set $\widehat{B}_i$ to be a ball centered at some point in $\partial W$ and with radius at most
twice the radius of $B_i$, such that $B_i\cap\partial W\subset\widehat{B}_i\subset 3B_i$; then
\begin{align*}
\sum_{i\in I}\frac{\mu_W(B_i)}{\rad(B_i)^\sigma}
\ge C_d^{-2}\, \sum_{i\in I_1}\frac{\mu_W(\widehat{B}_i)}{\rad(\widehat{B}_i)^\sigma}
\ge (C_d^2C)^{-1}\, \sum_{i\in I} \mu_0(\widehat{B}_i)&\ge (C_d^2C)^{-1}\, \mu_0(B(\xi,r)\cap\partial W)\\
&=(C_d^2C)^{-1}\mu_0(B(\xi,r)).
\end{align*}
Taking the infimum over all such covers and then letting $\eps\to 0^+$, we see that
\[
\Ha_{\mu_W}^{-\sigma}(B(\xi,r))\ge C^{-1}\, \mu_0(B(\xi,r)).
\]
On the other hand, for $\eta>0$ there exists $\eps_\eta>0$ so that for each $0<\eps<\eps_\eta$ 
\[
\Ha_{\mu_W}^{-\sigma}(B(\xi,r))=\Ha_{\mu_W}^{-\sigma}(B(\xi,r)\cap\partial W)
\le \Ha^{-\sigma}_{\mu_W,\eps}(B(\xi,r)\cap\partial W)+\eta.
\]
We ensure that $\eps<r$ as well, and find a cover of $B(\xi,r)$ by balls $B_i$, $i\in I\subset\N$,
each centered at points in $B(\xi,r)\cap\partial W$ with 
radius $\eps$, so that the balls $\tfrac{1}{5}B_i$ are pairwise disjoint. This is always possible because of the doubling
property of $\overline{W}$, see for example~\cite{HKST}.
Note that $B_i\subset B(\xi,2r)$.
It follows that 
\[
\Ha_{\mu_W}^{-\sigma}(B(\xi,r))-\eta\le\, \sum_{i\in I}\frac{\mu_W(B_i)}{\rad(B_i)^\sigma}
\le C_d^3\, \sum_{i\in I} \frac{\mu_W(\tfrac15B_i)}{5^\sigma\rad(\tfrac15B_i)^\sigma}
\le C\, \frac{C_d^3}{5^\sigma}\, \sum_{i\in I}\mu_0(\tfrac15B_i)\le C\, \frac{C_d^3}{5^\sigma}\, \mu_0(B(\xi,2r)).
\]
Letting $\eta\to 0^+$ and then applying the doubling property of $\mu_0$ yields
\[
\Ha_{\mu_W}^{-\sigma}(B(\xi,r))\le C\, \frac{C_d^3}{5^\sigma} \mu_0(B(\xi,2r))\le C\, \frac{C_d^4}{5^\sigma} \mu_0(B(\xi,r)).\qedhere
\]
\end{proof}

\subsection{Potential theory}\label{sec:PotThry}

Let $(W,d_W,\mu_W)$ be a metric measure space, and let $1\le p<\infty.$  Given a family $\Gamma$ of non-constant, compact, rectifiable curves in 
$W$, we define the \emph{$p$-modulus} of $\Gamma$ by
\begin{equation*}
\Mod_p(\Gamma)=\inf_\rho\int_W\rho^pd\mu_W,
\end{equation*}
where the infimum is taken over all Borel functions $\rho:W\to[0,\infty]$ such that $\int_\gamma \rho\,ds\ge 1$ for all $\gamma\in\Gamma$.  Given a function $u:W\to\overline\R$, we say that a Borel function $g:W\to[0,\infty]$ is an \emph{upper gradient} of $u$ if the following holds for all non-constant, compact, rectifiable curves $\gamma:[a,b]\to W$:
\begin{equation*}
|u(\gamma(b))-u(\gamma(a))|\le\int_\gamma g\,ds,
\end{equation*} 
whenever $u(\gamma(a))$ and $u(\gamma(b))$ are both finite, and $\int_\gamma g\,ds=\infty$ otherwise.  We say that $g$ is a \emph{$p$-weak upper gradient of $u$} if the $p$-modulus of the family of curves where the above inequality fails is zero. The notion of
upper gradients first appeared in~\cite{HK}, and the notion of weak upper gradients first appeared in~\cite{HaK}; interested readers can find more
on these topics from~\cite{HKST}.

For $1\le p<\infty$, we define $\wtil N^{1,p}(W,\mu_W)$ to be the class of all functions in $L^p(W,\mu_W)$ which have an upper gradient belonging to $L^p(W,\mu_W)$.  Defining
\begin{equation*}
\|u\|_{\wtil N^{1,p}(W,\mu_W)}:=\|u\|_{L^p(W,\mu_W)}+\inf_g\|g\|_{L^p(W,\mu_W)},
\end{equation*}
where in the infimum is taken over all upper gradients $g$ of $u$, we then define an equivalence relation in $\wtil N^{1,p}(W,\mu_W)$ 
by $u\sim v$ if and only if $\|u-v\|_{\wtil N^{1,p}(W,\mu_W)}=0$.  The \emph{Newton-Sobolev space} $N^{1,p}(W,\mu_W)$ 
is then defined to be $\wtil N^{1,p}(W,\mu_W)/\sim$, equipped with the norm $\|\cdot\|_{N^{1,p}(W,\mu_W)}:=\|\cdot\|_{\wtil N^{1,p}(W,\mu_W)}$.  
Given an open set $\Omega\subset W$, one can similarly define $N^{1,p}(\Omega,\mu_W)$. 

When $u\in N^{1,p}(W,\mu_W)$, there is a distinguished $p$-weak upper gradient, denoted $g_u$, of $u$, such that
$g_u\in L^p(W,\mu_W)$ and $g_u\le g$ for every other $p$-weak upper gradient $g\in L^p(W,\mu_W)$ of $u$. This $g_u$ is a local object in the 
sense that if $u\in N^{1,p}(W,\mu_W)$ and $\Om$ is a nonempty open subset of $W$, then the minimal $p$-weak upper gradient of $u\vert_\Om$
is $g_u\vert_\Om$.
For more on Newton-Sobolev spaces and 
upper gradients, we refer the interested reader to~\cite{BB, HKST}. 

For $0<\theta<1$, $1\le p<\infty$, and a function $u\in L^1_\loc(W,\mu_W)$, we define
\begin{equation*}
\|u\|_{B^\theta_{p,p}(W,\mu_W)}^p:=\int_W\int_W\frac{|u(y)-u(x)|^p}{d(x,y)^{\theta p}\mu_W(B(x,d(x,y)))}d\mu_W(y)d\mu_W(x).
\end{equation*}
We then define the \emph{Besov space} $B^\theta_{p,p}(W,\mu_W)$ as the set of all functions in $L^p(W,\mu_W)$ for which the above seminorm is finite.
The theory of Besov spaces has a large volume of literature, see for instance~\cite{Gagliardo, Hitchhiker, JW}
and the references therein for the Euclidean setting, and~\cite{GriHuLau, BouPaj, TriYan, Kaimanovich, GoKoSh}
for the setting of metric measure spaces equipped with a doubling measure. 
Since the extant literature on this topic is vast, we cannot hope to list them all here, and the references
we give here are merely a sampling.

We can now define the following capacities with respect to the Newton-Sobolev and Besov spaces.  For $E\subset W$, by $\BCap_p^W(E)$ we mean the number 
\[
\BCap_p^W(E):=\inf_u\left(\int_W|u|^p\, d\mu_W+\inf_g\int_Wg^p\, d\mu_W\right),
\]
where the first infimum is over all functions $u\in N^{1,p}(W,\mu_W)$
that satisfy $u\ge 1$ on $E$, and the second infimum is over all upper gradients $g$ of $u$.  By $\BCap^W_{\theta,p}(E)$ we mean the number
\[
\BCap_{\theta,p}^W(E):=\inf_u\left(\int_W|u|^p\, d\mu_W+\|u\|_{B^\theta_{p,p}(W,\mu_W)}^p\right),
\]
where the infimum is over all functions $u\in B^\theta_{p,p}(W,\mu_W)$ that satisfy $u\ge 1$ on a
neighborhood of $E$ in $W$.   We will also use the following variational (Newton-Sobolev) capacity: given an open set $\Omega\subset W$ and a set $E\subset\Omega$,  by $\vcap^W_p(E,\Om)$ we mean the number 
\[
\vcap_p^W(E,\Omega):=\inf_u\int_\Omega g_u^p\,d\mu_W,
\]
where the infimum is taken over all $u\in N^{1,p}(W,\mu_W)$ such that $u\ge 1$ on $E$ and $u=0$ in $W\setminus\Omega$.  In this paper, we will often consider Newton-Sobolev and Besov spaces as defined with respect to different metric 
measure spaces.  For this reason, we have kept the dependence on $W$ in the above notation.  

We say that a property holds $\BCap_p^W$-q.e.~if the set of all points for which the property fails is a $\BCap_p^W$-capacitary null set.
Similarly, we say that a property holds $\BCap_{\theta,p}^W$-q.e.~if the set of all points for which the property fails is a $\BCap_{\theta,p}^W$-capacitary null set.

Let $1\le q, p<\infty$.  Following~\cite{HaK}, we say that 
$(W,d_W,\mu_W)$ supports a \emph{$(q,p)$-Poincar\'e inequality} if there exist constants $C\ge 1$ and $\lambda\ge 1$ such that the following holds for balls $B\subset W$ and all function-upper gradient pairs $(u,g)$:
\begin{equation*}
\left(\fint_B|u-u_B|^qd\mu_W\right)^{1/q}\le C\rad(B)\left(\fint_{\lambda B}g^p\,d\mu_W\right)^{1/p}.
\end{equation*}
Here and throughout this paper, we use the notation
\[
u_B=\fint_B u\,d\mu_W=\frac{1}{\mu_W(B)}\int_Bu\,d\mu_W.
\]
When $(W,d_W,\mu_W)$ is a geodesic space which supports a $(q,p)$-Poincar\'e inequality, then we can take $\lambda=1$,
as shown in~\cite[Corollary~9.8]{HaK}. 
For more on Poincar\'e inequalities we refer the interested reader to~\cite{HKST}.

It was also shown in~\cite[Theorem~5.1]{HaK} that for doubling metric measure spaces,
the validity of $(1,p)$-Poincar\'e inequality implies the validity of $(p,p)$-Poincar\'e inequality. Putting this together
with~\cite[Theorem~5.53]{BB}, we obtain
the following Maz'ya-type inequality, see for example \cite[Theorem~5.53]{BB}:

\begin{thm}\label{thm:Mazya}
Assume that $(W,d_W,\mu_W)$ is a doubling geodesic space and that it
supports a $(1,p)$-Poincar\'e inequality for some $1\le p<\infty$.  For $u\in N^{1,p}_\loc(W,\mu_W)$, let $S:=\{x\in W: u(x)=0\}$.  
Then for all balls $B\subset W$, 
\begin{equation*}
\fint_B|u|^pd\mu_W\le\frac{C(\rad(B)^p+1)}{\BCap^W_p(B\cap S)}\int_{2B}g_u^pd\mu_W.
\end{equation*}
Here $g_u$ is the minimal $p$-weak upper gradient of $u$. 
\end{thm}


\subsection{Uniform domains}\label{uniformity}

The notion of uniform domains first arose in the study of quasiconformal mapping theory, as see for example~\cite{GerOsg, MarSar, GerMar, BHK},
but since then has also played crucial roles in Euclidean potential theory
see for example~\cite{Jones, Aik1, Aik2, KRZ} and 
potential theory in metric measure spaces~\cite{Maly,BS, BBS}.
It is also worth noting that in a geodesic metric space, every domain can be approximated 
by uniform domains, as proved in~\cite{TRaj}.

\begin{defn}\label{def:uniform}
A noncomplete, locally compact metric space $(W,d_W)$ is said to be a uniform domain if there is a constant $A>1$ such that, with
$\partial W=\overline W\setminus W$, for each pair of points $x,y\in W$ we can find a curve $\gamma$ in $W$ with end points $x,y$
such that the length $\ell(\gamma)\le A\, d_W(x,y)$ and in addition, for each point $z$ in the trajectory of $\gamma$, we have
\[
d(z,\partial W):=\dist(z,\partial W)\ge A^{-1}\, \min\{\ell(\gamma_{x,z}), \ell(\gamma_{z,w})\},
\]
where $\gamma_{x,z}$ and $\gamma_{z,w}$ denote each subcurve of $\gamma$ with end points $x,z$ and $z,w$ respectively. 
Curves $\gamma$ satisfying the above condition are called uniform curves.
\end{defn}

The potential-theoretic utility of uniform domains comes from the ability to connect pairs of points in the domain by chains of Whitney-type
balls.

\begin{lem}\cite[Section~3.1]{GS}\label{lem:Chain}
	Let $\Omega\subset W$ be a uniform domain.  Then for each $w,z\in\partial\Omega$, there exists a uniform curve $\gamma$ joining $w$ and $z$ and a chain of balls $\{B_k:=B(x_k,r_k)\}_{k\in\Z}$ covering $\gamma$ such that
	\begin{enumerate}
		\item $\lim_{k\to+\infty} x_k=w$ and $\lim_{k\to-\infty}x_k=z$.
		\item If $k\ge 0$ and $x\in B_k$, then $r_k\simeq d(x,\partial\Omega)\simeq d(x,w)\lesssim 2^{-|k|/(4C)} d(w,z)$.
		\item If $k<0$ and $x\in B_k$, then $r_k\simeq d(x,\partial\Omega)\simeq d(x,z)\lesssim 2^{-|k|/(4C)} d(w,z)$.
		\item The collection $\{4B_k\}_{k\in\Z}$ has bounded overlap.
	\end{enumerate}
\end{lem}

\subsection{Compact doubling metric measure spaces as boundaries of uniform domains of globally controlled geometry}\label{sec:HypFill}

From~\cite{BBS} we know that every compact doubling metric measure space $(Z,d,\nu)$ (up to a biLipschitz change in the metric) is the boundary
of a bounded uniform domain. We denote this uniform domain $X_\eps$, as it is obtained as the uniformization of a hyperbolic filling graph of
$Z$ with parameter $\eps>0$. In~\cite{BBS}, a one parameter family of measures $\mu_\beta$, $\beta>0$, was also constructed on $X_\eps$,
based on the measure $\nu$ on $Z$, such that the measure $\nu$ is $\beta/\eps$--codimensional with respect to $\mu_\beta$.

It was also shown in~\cite{BBS} that the measure $\mu_\beta$ is doubling on the metric space $(X_\eps,\, d_\eps)$ and that
the metric measure space $(X_\eps,d_\eps,\mu_\beta)$ supports a $(1,1)$-Poincar\'e inequality, which, thanks to the H\"older inequality,
is the strongest of all the $(1,p)$-Poincar\'e inequalities.

The notion of hyperbolic filling was proposed by Gromov~\cite{Gro}, and developed further in~\cite{Car, BouPaj, BonKle,  BuySch, BonSak, BonSakSot}.
In this subsection we describe a construction of the so-called hyperbolic filling of a compact doubling metric measure space $(Z,d,\nu)$ as
given in~\cite{BBS}.

Given a compact doubling metric measure space $(Z,d,\nu)$, 
by rescaling the metric if necessary (that is, replacing $d$ with $(2\, \diam(Z))^{-1} d$), in constructing the hyperbolic filling we may assume
that $\diam(Z)<1$. 
We fix a point $z_0\in Z$ and set $A_0=\{z_0\}$. We also fix $\alpha>2$
and $\tau>2$, and for each positive integer $n$ we choose a maximal $\alpha^{-n}$-separated subset $A_n$ of $Z$ such that $A_{n-1}\subset A_n$.
We then set 
\[
V=\bigcup_{n=0}^\infty A_n\times \{n\},
\]
and the set $V$ forms the vertex set for the graph $G$ we construct next. For $(\xi,n), (\zeta,m)\in V$, we say that $(\xi,n)$ is a 
neighbor of $(\zeta,m)$,
if $|n-m|\le 1$ and in addition, either $n=m$ and $B_Z(\xi,\alpha^{-n})$ intersects $B_Z(\zeta,\alpha^{-n})$, or else $n=m\pm 1$ and
$B_Z(\xi,\tau\, \alpha^{-n})$ intersects $B_Z(\zeta, \tau\, \alpha^{-m})$. 
We turn $G$ into a metric graph $X$ by gluing unit-length intervals $[v,w]$ between
each pair of neighboring vertices $v, w$. Let $d_X$ denote the resulting graph metric on $X$.

The construction of $X$, as given here, is from~\cite{BBS}, and is a simple modification of that given 
in~\cite{BouPaj, BuySch, BonSak, BonSakSot}, and it is known that $X$ is a Gromov hyperbolic space with
visual boundary that is quasisymmetric to the metric space $Z$. Motivated 
by~\cite{BHK}, we use the tool of uniformization of the metric on $X$ as follows.
Let $\eps=\log\alpha$. With $v_0=(z_0,0)$ playing the role of a root vertex, when $x,y\in X$ we set
\[
d_\eps(x,y)=\inf_\gamma\int_\gamma e^{-\eps\, d_X(\gamma(t),v_0)}\, ds(t),
\]
where the infimum is over all paths $\gamma$ in $X$ with end points $x,y$. We also lift up the measure $\nu$ on $Z$ to $X_\eps$ as follows.
For each edge $[v,w]$ with $v=(\xi_v,n_v)$ and $(\xi_w,n_w)=w\in V$, we set 
\[
a_{(v,w)}=\frac{\nu(B_Z(\xi_v,\alpha^{-n_v}))+\nu(B_Z(\xi_w,\alpha^{-n_w}))}{2},
\] 
and then for $A\subset X_\eps$ we set
\[
\mu_\beta(A):=\int_A \left(\sum_{(v,w)}a_{(v,w)}\chi_{[v,w]}(x)\, e^{-\beta d_X(x,v_0)}\right) d\mathcal{H}^{1}(x).
\]

The results from~\cite{BBS} related to the construction $(X_\eps,d_\eps,\mu_\beta)$ are summarized below.

    \begin{thm}\cite[Theorem~1.1]{BBS}\label{thm:HypFillThm}
    Let $(Z,d,\nu)$ be a compact, doubling metric measure space, let $1\le p<\infty$, and let $\alpha,\,\tau>1$ be the parameters from the above construction.  Let $\eps=\log\alpha.$ Then for each $\beta>0$, the uniformized hyperbolic filling $(\overline X_\eps,d_\eps,\mu_\beta)$ constructed above satisfies the following:
    \begin{enumerate}
        \item $(Z,d)$ is biLipschitz equivalent to $(Z,d_\eps)$, with the biLipschitz constant depending on $\eps$ and $\diam(Z)$.
        \item Both $(X_\eps,d_\eps,\mu_\beta)$ and $(\overline X_\eps,d_\eps,\mu_\beta)$ are doubling and support a $(1,1)$-Poincar\'e inequality;
        moreover, we have $N^{1,p}(X_\eps,\mu_\beta)=N^{1,p}(\overline X_\eps,\mu_\beta)$.
        \item For all $z\in Z$ and $0<r\le2\diam(Z)$, 
        \[
        \nu(B_\eps(z,r)\cap Z)\simeq\frac{\mu_\beta(B_\eps(z,r))}{r^{\beta/\eps}},
        \]
        that is, $\nu$ is $\beta/\eps$-codimensional with respect to the measure $\mu_\beta$.
        \item When $\beta/\eps<p$, there exist bounded linear trace and extension operators
        \[
        T_X:N^{1,p}(\overline X_\eps,\mu_\beta)\to B^{1-\beta/(\eps p)}_{p,p}(\overline Z,\nu)\quad\text{and}\quad E_X:B^{1-\beta/(\eps p)}_{p,p}(\overline Z,\nu)\to N^{1,p}(\overline X_\eps,\mu_\beta)
        \]
        such that $T_X\circ E_X=\Id$. The boundedness of $T_X$ and $E_X$ are with respect to the full norms of the respective spaces
        (which includes the $L^p$-norms), and in addition, we also get boundedness in energy as well.
        Furthermore, if $u\in B^{1-\beta/(\eps p)}_{p,p}(\overline Z,\nu)$, then $\nu$-a.e.\ $z\in \overline Z$ is 
        a Lebesgue point of $E_Xu$. 
    \end{enumerate}
    In all of the above, the constants depend only on  $\alpha$, $\tau$, $\beta$,  $p$, and the doubling constant of $\nu$.   
\end{thm}

Indeed, an explicit construction is given as follows. For $w\in N^{1,p}(\overline{X}_\eps,\mu_\beta)$ and $\xi\in Z$, we have
\begin{equation}\label{eq:Tx-def}
T_Xw(\xi)=\lim_{r\to 0^+}\jint_{B(\xi,r)}w\, d\mu_\beta,
\end{equation}
while, for $u\in B^{1-\beta/(\eps p)}_{p,p}(Z,\nu)$ and vertices $v\in X$, we have
\begin{equation}\label{eq:Ex-def}
E_Xu(v)=\jint_{D(v)}u\, d\nu,
\end{equation}
and then extend $E_Xu$ linearly to edges in $X$. Here $D(v)=B_Z(\xi_v,\alpha^{-n_v})$ when $v=(\xi_v,n_v)$.  

As a consequence of Theorem~\ref{thm:HypFillThm}, it was also shown in \cite[Proposition~13.2]{BBS} that  
when $\beta/\eps<p$ and $\theta=1-\beta/(\eps p)$, for $E\subset\overline Z$
we have that 
\begin{equation}\label{eq:CapComparison}
\BCap^{\overline Z}_{\theta,p}(E)\simeq\BCap^{\oXeps}_p(E).
\end{equation}

Now we discuss some notational caveats. The primary focus of this paper is a metric measure space $(Z,d,\nu)$ with 
$(Z,d)$ locally compact and bounded but not complete, and $\nu$ a doubling measure on $Z$. Thus the hyperbolic filling
uniform domain is associated not with $Z$ but its metric completion $\overline Z$. Thus $\partial \overline{X}_\eps=\overline Z$.
Since $Z$ is locally compact, it follows that $Z$ is relatively open in $\overline Z$. Moreover, as $X_\eps$ is a uniform domain
in the sense of Definition~\ref{def:uniform}, so is $\overline{X}_\eps\setminus\partial Z$.

For $x\in\overline{X}_\eps$ and $r>0$, by $B(x,r)$ we mean the ball centered at $x$, with radius $r$ in
	the metric $d_\eps$. At various stages we might want to focus on balls in $X_\eps$, $\overline{X}_\eps$, $Z$, or $\partial Z$, but    
	these are intersections of balls in $\overline{X}_\eps$ with the relevant subsets, and the need to distinguish these balls arise only
	when integrating over such balls; as the integrals are with respect to measures supported on these various subsets, the intersection
	of the ball in $\overline{X}_\eps$ with these subsets will be replaced here by the measures in the relevant integrals.

\subsection{Cheeger differential structure and the definition of differential $\nabla f$}\label{sub:diff-struct}

From~\cite{Chee} we know that whenever $(W,d_W,\mu_W)$ is a doubling metric measure space supporting a $(1,p)$-Poincar\'e 
inequality, there is a positive integer $N$ and a countable family of measurable sets $W_\alpha\subset W$ with $\mu_W(W_\alpha)>0$,
and for each $\alpha$ a Lipschitz map $\Phi_\alpha:W\to \R^N$, and a measurable inner product 
$W_\alpha\ni w\mapsto\langle\cdot,\cdot\rangle_w$
satisfying the following condition. Whenever $f:W\to\R$ is
a Lipschitz function, there is a measurable function $D_\alpha f:W_\alpha\to\R^N$ such that for $\mu_W$-a.e.~$w\in W_\alpha$ we have
\[
\limsup_{y\to w}\frac{f(y)-f(w)-\langle D_\alpha f(w),\Phi_\alpha(y)-\Phi_\alpha(w)\rangle_w}{d(y,w)}=0.
\]
Moreover $D_\alpha$ is a linear operator on the class of all Lipschitz functions on $W$, and 
\[
\langle D_\alpha f(w), D_\alpha f(w)\rangle_w\approx g_f(w)^2
\]
for $\mu_W$-a.e.~$w\in W_\alpha$; such a differential structure extends also to functions in $N^{1,p}(W,\mu_W)$.

In the setting where $W=X_\eps$, $d_W=d_\eps$, and $\mu_W=\mu_\beta$, there is a natural differential structure 
on $(X_\eps, d_\eps, \mu_\beta)$, for $X_\eps$ is a graph with edges that are intervals \emph{of varying length} (from the
uniformization procedure conducted to obtain $d_\eps$). We know that
each edge in $X_\eps$ forms a singleton family of paths in $X_\eps$ with positive $p$-modulus; hence
when $f\in N^{1,p}(\overline X_\eps,\mu_\beta)$, we know that $f$ is absolutely continuous on each such edge and hence is
differentiable there. In this case, we can set $Df(x)=f^\prime(x)$ for $\mathcal{H}^1$-almost every $x$ in that
edge by first choosing an orientation for that edge. The choice of orientation does not matter as long as
we consistently use the same orientation for all $f\in N^{1,p}(X_\eps,\mu_\beta)$; as $\mu_\beta(\partial X_\eps)=0$,
we know that such $\nabla f$ is well-defined $\mu_\beta$-a.e.~in $\overline{X}_\eps$. So we set 
$\nabla f(x):=f^\prime(x)$ whenever $x$ belongs to an edge and $f^\prime(x)$ exists; $\mu_\beta$-almost every
$x\in X_\eps$ is such a point. Thus $\nabla f$ is well-defined with $\nabla f:X_\eps\to\R$, with
$|\nabla f|^2=g_f^2=\langle \nabla f,\nabla f\rangle$. This differential structure is a special choice of a Cheeger
differential structure, where each $W_\alpha$ is an edge of $X_\eps$, and $\Phi_\alpha$ measures
the location of $x\in W_\alpha$ from one end of the edge (which also then determines the orientation of that edge)
with respect to the metric $d_\eps$. The fact that this structure satisfies the limit that defines the Cheeger differential
structure condition is merely Taylor's theorem for the interval.

In the setting of infinitesimally Hilbertian spaces (spaces where minimal weak upper gradients can be endowed with an inner product 
structure that agrees with the pointwise values of the upper gradient) as in~\cite{Gigli} (see also~\cite{AGS} for a nice
exposition on the topic of test plans and application of Wasserstein spaces), we can replace the Cheeger structure $Df$ with the
upper gradient structure. This is because under the assumption of infinitesimal Hilbertianity there is an Euler-Lagrange equation associated
with the upper gradient-based energy minimization problem. In this paper we focus on the natural differential structure $\nabla f$,
keeping in mind such a flexibility beyond the hyperbolic filling setting.

\begin{defn}\label{def:p harmonic ext}
Recalling the construction of the uniformization of the hyperbolic filling $X_\eps$ from 
Subsection~\ref{sec:HypFill}, we say that a function $\overline{u}\in N^{1,p}(X_\eps,\mu_\beta)=N^{1,p}(\overline{X}_\eps,\mu_\beta)$ is the
$p$-harmonic extension of a function $u\in B^\theta_{p,p}(Z,\nu)$ if $T_X(\overline{u})=u$ and whenever 
$v\in N^{1,p}(X_\eps,\mu_\beta)$ with $T_Xv=0$, we have 
\[
\int_{X_\eps}|\nabla \overline{u}|^{p-2}\langle \nabla\overline{u},\nabla v\rangle\, d\mu_\beta=0.
\]
\end{defn}

Given that functions in $B^\theta_{p,p}(Z,\nu)$ have an extension to $X_\eps$ that lies in $N^{1,p}(X_\eps,\mu_\beta)$,
such $p$-harmonic extensions always exist, and given the support of $p$-Poincar\'e inequality on $(X_\eps,d_\eps,\mu_\beta)$, we 
also have the uniqueness of such an extension, see~\cite[Theorem~7.2]{BB} or~\cite[Theorem~7.8, Theorem~7.14]{Chee}.

\subsection{Standing assumptions and statement of the Dirichlet problem for fractional operators on $Z$}\label{sub:standing}
In this final subsection of Section~2, we state the standing assumptions we will operate under for the rest of the paper, and 
state the Dirichlet problem that is the principal focus of the paper.

\noindent \emph{Standing assumptions:} In this paper, $(Z,d,\nu)$ is a locally compact, non-complete bounded metric 
measure space with $\nu$ a doubling measure on $Z$, and we equip $\partial Z:=\overline{Z}\setminus Z$ with the codimension
$\sigma$ Hausdorff measure with respect to $\nu$, namely $\pi:=\mathcal{H}^{-\sigma}_\nu$, and we extend $\nu$ to $\partial Z$ 
so that $\nu(\partial Z)=0$, and extend $\pi$ to $Z$ by setting $\pi(Z)=0$. 
We will also assume that $0<\pi(\partial Z)<\infty$.

We fix $\eps>\log 2$. Then, for each $\beta>0$ consider
$(X_\eps,d_\eps,\mu_\beta)$ to be the uniform domain with $\mu_\beta$ doubling on $X_\eps$ and supporting a $(1,1)$-Poincar\'e inequality,
and (by perturbing the metric on $Z$ in a biLipschitz fashion if necessary) so that $\overline{Z}=\partial X_\eps=\overline{X}_\eps\setminus X_\eps$.

Since $Z$ is locally compact, we know that $Z$ is an open subset of 
the compact metric space $\overline{Z}$. Observe that $\overline{Z}$ is compact because the measure $\nu$ on $Z$
is doubling, and so the zero-extension of $\nu$ to $\partial Z$ is also doubling; bounded complete doubling metric measure spaces
are necessarily compact, see for example~\cite[Lemma~4.1.14]{HKST}. 

\vskip .2cm


We now outline the construction of the nonlocal operator which forms the basis of our problem, as introduced in \cite{CKKSS}.

\begin{defn}\label{def:ET}
We fix $1<p<\infty$ and $0<\theta<1$, and choose $\beta=\eps\, (1-\theta)\, p$.
This choice leads to the Besov space $B^\theta_{p,p}(Z,\nu)$ as the trace space of the Newton-Sobolev space
$N^{1,p}(\overline X_\eps,\mu_\beta)$.  We define the operator $\mathcal{E}_T:B^\theta_{p,p}(\overline Z,\nu)\times B^\theta_{p,p}(\overline Z,\nu)\to\R$, induced from this relationship, by 
\[
\mathcal{E}_T(u,v):=\int_{X_\eps}|\nabla \overline{u}|^{p-2}\langle \nabla \overline{u},\, \nabla \overline{v}\rangle\, d\mu_\beta,
\]
where $\overline{u}$ and 
$\overline{v}$ are the $p$-harmonic extensions of $u$ and $v$ respectively to $X_\eps$ given by Definition~\ref{def:p harmonic ext}.
\end{defn}

The operator $\mathcal{E}_T$ has the following properties:

\begin{enumerate}
\item $\mathcal{E}_T(u,v+w)=\mathcal{E}_T(u,v)+\mathcal{E}_T(u,w)$,
\item $\mathcal{E}_T(0,v)=\mathcal{E}_T(u,0)=0$,
\item $\mathcal{E}_T(\alpha u,\beta v)=|\alpha|^{p-2}\alpha\beta\, \mathcal{E}_T(u,v)$ whenever $\alpha,\beta\in\R$,
\item $\mathcal{E}_T(u,u)\ge 0$, and $\mathcal{E}_T(u,u)=0$ if and only if $u$ is constant,
\item We also have the comparison
\[
\|u\|_{B^\theta_{p,p}(\overline Z,\nu)}^p:=\int_Z\int_Z\frac{|u(y)-u(x)|^p}{d(x,y)^{\theta p}\, \mu(B(y,d(x,y)))}\, d\nu(y)\, d\nu(x)
\approx \mathcal{E}_T(u,u).
\]
\end{enumerate}


\begin{remark}\label{rem:Et-2nd-trace}
In the definition of $\mathcal{E}_T$, we do not have to consider the $p$-harmonic extension of the second function, $v$; indeed,
any choice $\widehat{v}$ of extension of $v$ to $X_\eps$ would do as long as $T_X\widehat{v}=v$. This is because, as $\overline{u}$ itself
is $p$-harmonic in the domain $X_\eps$, we have that
\[
\int_{X_\eps}|\nabla \overline{u}|^{p-2}\langle \nabla \overline{u},\, \nabla(\widehat{v}-\overline{v})\rangle \, d\mu_\beta=0.
\]	
\end{remark}

In what follows, $T_X:N^{1,p}(\overline X_\eps,\mu_\beta)\to B^\theta_{p,p}(\overline Z,\nu)$ and
$T_Z:B^\theta_{p,p}(\overline Z,\nu)\to B^{\theta-\sigma/p}_{p,p}(\partial Z, \pi)$ are the relevant trace operators,
with $T_X$ as described in Theorem~\ref{thm:HypFillThm} above, and $T_Z$ as described in
Lemma~\ref{lem:traceZ} below. In particular, for $u\in B^\theta_{p,p}(\overline Z,\nu)$, and for $\zeta\in\partial Z$,
the number $T_Zu(\zeta)$, if it exists, satisfies 
\begin{equation}\label{eq:trace-defn}
\limsup_{r\to 0^+}\jint_{B(\zeta,r)}|u-T_Zu(\zeta)|\, d\nu=0.
\end{equation}

We will show below that if the measure $\pi=\mathcal{H}^{-\sigma}_\nu$ on $\partial Z$ 
is $\sigma$-codimensional with respect to $\nu$, then the trace operator $T_Z$
also exists, see the discussion in Section~\ref{Sec:4}. However, to discuss Besov functions on $Z$ with null
trace on $\partial Z$, we do not need $\pi$ to be $\sigma$-codimensional, see Section~\ref{Sec:3}.

The goal of this paper is to demonstrate the existence of solutions to the Dirichlet boundary value problem 
on $(Z,d,\nu)$ with boundary data defined on $\partial Z$, and prove stability properties of this problem. 
We now pose the problem studied here.

 
\begin{defn}\label{def:Dirichlet1}
Given $G\in L^{p'}(Z,\nu)$ with $p'=p/(p-1)$ 
and a function $f\in B^{\theta-\sigma/p}_{p,p}(\partial Z,\pi)$,
we say that $u\in B^\theta_{p,p}(\overline Z,\nu)$ is a solution to the \emph{$(f,G)$-inhomogeneous Dirichlet problem} if 
\begin{align*}
\mathcal{E}_T(u,v)=&\int_Z G\, v\, d\nu \text{ whenever }v\in B^\theta_{p,p}(\overline Z,\nu) \text{ with }T_Z v=0\ \ \text{$\pi$-a.e. in }\partial Z,\\
T_Zu=&f \ \ \text{$\pi$-a.e. in }\partial Z.
\end{align*}
\end{defn}

If $G\equiv 0$, then we call the $(f,G)$-inhomogeneous Dirichlet problem also as the \emph{homogeneous Dirichlet problem with
boundary data $f$}.

\begin{remark}\label{rem:Euler-Lagrange}
A function $u$ solves the above Dirichlet problem if and only if it is a minimizer of the energy 
\begin{equation}\label{eq:EnergyFunctional}
I_G(v):=\int_{\overline X_{\eps}}|\nabla \overline{v}|^p\,d\mu-p\int_Z vG\, d\nu
\end{equation}
over all functions $v\in B^\theta_{p,p}(\overline Z,\nu)$ with trace $T_Zv=f$ on $\partial Z$,
because the weak equation $\mathcal{E}_T(u,v)=\int_Zv\,G\, d\nu$ is the Euler-Lagrange equation corresponding to the 
minimization of the energy $I_G$. Recall that $\overline v$ denotes the $p$-harmonic extension of $v$ to $X_\eps$.
\end{remark}

\begin{remark}
Existence of the operator $T_Z$ is known only under certain circumstances, namely that the measure $\pi$ on
$\partial Z$ is $\sigma$-codimensional with respect to the measure $\nu$ on $Z$, see~\eqref{eq:codim-def} (with $s=\sigma$). 
However,even in the classical Euclidean setting, 
if $\partial Z$ does not support such a codimensional measure with respect to $\nu$, then neither the existence of 
such a trace operator, nor its boundedness, is known. In this case, we can consider a weaker boundary value problem, given
as follows. With the ``boundary data'' $f$ now belonging to $B^\theta_{p,p}(\overline Z,\nu)$, we want to find
$u\in B^\theta_{p,p}(\overline Z,\nu)$ such that
\begin{align*}
\mathcal{E}_T(u,v)=&\int_Z G\, v\, d\nu \text{ whenever }v\in B^\theta_{p,p}(\overline Z,\nu) \text{ with }T_Z v=0\, \text{$\pi$-a.e. in }\partial Z,\\
u-f\in & B^\theta_{p,p,0}(Z,\nu),
\end{align*}
where $B^\theta_{p,p,0}(Z,\nu)$ is the closure of the collection of functions from $B^\theta_{p,p}(\overline Z,\nu)$ which have compact support in $Z$,
see Proposition~\ref{prop:1} below.
\end{remark}

\subsection{Examples}\label{sec:Examples}

Our results can be applied for any domain in a doubling metric measure space for which the restriction of the ambient measure is also doubling, and whose boundary is equipped with a measure which is codimensional with respect to the ambient measure.  While our results can therefore be applied in the setting of Riemannian manifolds, Carnot groups, and other metric measure spaces which support a Poincar\'e inequality, such as the Laakso space and Bourdon-Pajot space, our results do not require the assumption of a Poincar\'e inequality.  As such, they are also applicable in numerous fractal-type spaces where Poincar\'e inequalities do not hold.  We now provide a few examples of such settings.


\begin{exa}[Rickman's rug]
Let $L$ be the standard vonKoch snowflake curve, which is of dimension $d_L:=\log(4)/\log(3)$. The metric space 
$Z=L\times (0,1)$ is equipped with the product measure $\mu$ obtained from the Hausdorff measure on $L$ with the Lebesgue measure
on $(0,1)$. Note that $\partial Z=L\times\{0\}\cup L\times\{1\}$ is equipped with the Hausdorff measure $\pi$ of $L$.
In this case $\mu$ is Ahlfors $d_L+1$-regular, and $\pi$ is Ahlfors $d_L$-regular, which is $\sigma$-codimensional to $\mu$
where $\sigma=1$. Our theorems apply here with $1<p<\infty$ and $1/p<\theta<1$.

On the other hand, if $L_1$ is the vonKoch snowflake curve without its two end points, and $Z=L_1\times[0,1]$, then
$\partial Z$ consists of two disjoint copies of the interval $[0,1]$, and hence in this case
$\sigma=d_L$. In this case, our theorems apply with $d_L<p<\infty$ and $d_L/p<\theta<1$.
\end{exa}

\begin{exa}[Fractal examples]
Fractals that arise from a self-similar construction form a rich subclass of examples for us. As described by
Hino and Kumagai in~\cite{HinoKumagai}, if we start with a finite collection of Euclidean similarities with the
same scaling constant that leads
to a fractal set as the compact invariant set, then a small subcollection of these similarities yields a compact
subset that is the boundary of the domain that is the original invariant set sans the subset. Both invariant sets
are Ahlfors regular, but with different dimensions. If there are $n$ number of similarities in the original collection
and $m$ number of similarities in the subcollection, and if the similarities have scaling $\lambda$ with $0<\lambda<1$,
then the co-dimensionality $\sigma=\log(n/m)/\log(1/\lambda)$. 

As a concrete example we may consider the Sierpinski carpet, which has $[0,1]\times\{0\}$ as one side-edge. With $Z$ the
carpet with this edge removed, we have that the boundary $\partial Z$ is this edge. As the dimension of the carpet 
is $\log(8)/\log(3)$ and the edge has dimension $1$, in this case we have that $\sigma=\log(8/3)/\log(3)$, and
so we are permitted to consider $1<p<\infty$ and $\sigma/p<\theta<1$.

Another concrete example is the pentagasket (called the pentakun in~\cite{HinoKumagai}), with a Koch-type fractal curve
forming part of the boundary. As explained in~\cite[Subection~5.2, page~604]{HinoKumagai}, the dimension of the
pentagasket is $\log(5)/\log(\alpha)$ with $\alpha=\tfrac{3+\sqrt{5}}{2}$, and the dimension of the Koch-type curve
is $\log(4)/\log(\alpha)$, leading us to the co-dimension $\sigma=\log(5/4)/\log(\alpha)<1$, and so again
we can consider $1<p<\infty$ and $\sigma/p<\theta<1$.

The paper~\cite{HinoKumagai} has a lovely illustration of these two concrete examples. That paper also considers 
a trace result of a Besov class from the larger fractal to the smaller fractal boundary, but the Besov class on the 
larger fractal is different than the one considered in our paper and is related more to the Korevaar-Schoen spaces
rather than Besov spaces, as discussed for example in~\cite{Bad1, AR-Bad1,MurShim}.
\end{exa}

\section{Besov functions with zero trace}\label{Sec:3}

To define what a solution to the problem of interest is, we need a solid foundation of the theory of functions in
$B^\theta_{p,p}(Z,\nu)=B^\theta_{p,p}(\overline{Z},\nu)$ that have zero trace at $\partial Z$, 
see Definition~\ref{def:Dirichlet1}. The focus of this section is to 
lay such a foundation. We do so with minimal assumptions on the boundary $\partial Z$ of $Z$ as such a foundation may be useful in
other contexts as well.

Recall the standing assumptions from subsection~\ref{sub:standing}. 
In the next section, we will assume that $\Ha^{-\sigma}_\nu$ is codimension $\sigma$ Ahlfors regular with respect to $\nu$;
but we do not need this strong assumption in this section. We merely continue to assume that $0<\pi(\partial Z)<\infty$ and that $\nu(\partial Z)=0$. 
Since $\nu(\partial Z)=0$, we also have that $B^\theta_{p,p}(Z,\nu)=B^\theta_{p,p}(\overline{Z},\nu)$.  Note that as $(Z,d,\nu)$ is locally compact, $Z$ is open in $\overline Z$. 

\begin{prop}\label{prop:1}
    Let $0<\theta<1$, and let $p>\max\{1,\sigma/\theta\}$.  For each $k\in\N$, let $u_k\in B^\theta_{p,p}(\overline Z,\nu)$ be compactly supported in $Z$, and suppose that there exists $u\in B^\theta_{p,p}(\overline Z,\nu)$ such that $u_k\to u$ in $B^\theta_{p,p}(\overline Z,\nu)$.  Then for 
$\pi$-a.e.\ $\zeta\in\partial Z$, we have
    \[
    \lim_{r\to 0^+}\fint_{B(\zeta,r)}|u|^pd\nu=0.
    \]
In particular, the number $T_Zu(\zeta)=0$ for $\pi$-a.e.~$\zeta\in\partial Z$.
\end{prop}

\begin{proof}
    Fix $\alpha>2$, $\eps=\log\alpha$, and $\beta=\eps\, (1-\theta)\, p$ in the construction of the uniform domain $(X_\eps,d_\eps,\mu_\beta)$ for which 
    $\overline{Z}=\partial X_\eps$.
By Theorem~\ref{thm:HypFillThm}, we have that $E_Xu_k,\,E_Xu\in N^{1,p}(\overline X_\eps,\mu_\beta)$.
 Moreover,  $E_Xu_k\to E_Xu$ in $N^{1,p}(\overline X_\eps,\mu_\beta)$, which follows from the boundedness and 
 linearity of the extension operator $E_X$. 
 Since the $u_k$ are compactly supported in $Z$, it follows 
 from the construction of the extension operator $E_X$, see~\eqref{eq:Ex-def},
 that the functions $E_Xu_k$ are compactly supported in $\overline X_\eps\setminus\partial Z$, and so $E_Xu_k\in N^{1,p}_0(\overline X_\eps\setminus \partial Z,\mu_\beta)$.  
Note also that $\overline{X}_\eps\setminus\partial Z$ is a uniform domain as well.
As $N^{1,p}_0(\overline X_\eps\setminus\partial Z,\mu_\beta)$ is a closed subspace of $N^{1,p}(\overline X_\eps,\mu_\beta)$, see \cite[Theorem~2.36]{BB} for example, it follows that $E_Xu\in N^{1,p}_0(\overline X_\eps\setminus\partial Z,\mu_\beta)$.  From \cite[Theorem~1.1]{KKST}, it then follows 
that for $\BCap^{\oXeps}_p$--q.e.\ $\zeta\in\partial Z$, we have
    \begin{equation}\label{eq:Eu Zero}
    \lim_{r\to 0^+}\fint_{B_\eps(\zeta,r)}|E_Xu|d\mu_\beta=0.
    \end{equation}
Under our hypotheses on $p$ and the construction of the measure $\pi$ supported on $\partial Z$, we will see towards the end of this
proof that $\BCap^{\oXeps}_p$-null subsets of
$\partial Z$ are necessarily $\pi$-null.
    By the construction of the trace operator $T_X:N^{1,p}(\overline X_\eps,\mu_\beta)\to B^\theta_{p,p}(\overline Z,\nu)$ given in 
    Theorem~\ref{thm:HypFillThm}, see~\eqref{eq:Tx-def}, the condition
    \eqref{eq:Eu Zero} implies that $T_X\circ E_Xu(\zeta)=0$. 

    Fix $r>0$, and let $\zeta\in\partial Z$ such that \eqref{eq:Eu Zero} holds.  By Theorem~\ref{thm:HypFillThm}, we have that $\nu$-a.e.\ $w\in B_Z(\zeta,r)$ is a Lebesgue point of $E_Xu$ with $T_X\circ E_Xu(w)=u(w)$.  As $X_\eps$ is a uniform domain in its completion, we can join such $w$ to $\zeta$ by a uniform curve, and obtain the corresponding chain of balls $\{B_k\}_{k\in\Z}$ given by Lemma~\ref{lem:Chain}.  We then have that 
    \begin{align*}
    |u(w)|=|T_X\circ E_Xu(w)-T_X\circ E_X u(\zeta)|
    \le\sum_{k\in\Z}|(E_Xu)_{B_k}-(E_Xu)_{B_{k+1}}|\lesssim\sum_{k\in\Z}r_k\left(\fint_{4B_k}g_{E_Xu}^pd\mu_\beta\right)^{1/p},
    \end{align*}
    where $r_k:=\rad(B_k)$, and we have applied the $(1,p)$-Poincar\'e inequality to obtain the last inequality. 

    Since $\beta/\eps=p(1-\theta)$, we can choose $\delta>0$ such that $\beta/\eps+\delta<p$.  By the H\"older inequality, we then have that 
    \begin{align*}
        |u(w)|&\lesssim\sum_{k\in\Z}r_k^{1-\frac{\beta/\eps+\delta}{p}}\left(r_k^{\beta/\eps+\delta}\fint_{4B_k}g_{E_Xu}^pd\mu_\beta\right)^{1/p}\\
        &\le\left(\sum_{k\in\Z}r_k^{\frac{p-\beta/\eps-\delta}{p-1}}\right)^{1-1/p}\left(\sum_{k\in\Z}r_k^{\beta/\eps+\delta}\fint_{4B_k}g_{E_Xu}^pd\mu_\beta\right)^{1/p}.
    \end{align*}
By Lemma~\ref{lem:Chain}, we have that 
\[
\sum_{k\in\Z}r_k^{\frac{p-\beta/\eps-\delta}{p-1}}\simeq d(w,\zeta)^{\frac{p-\beta/\eps-\delta}{p-1}}\sum_{k\in\Z}2^{-|k|\frac{p-\beta/\eps-\delta}{p-1}}\simeq d(w,\zeta)^{\frac{p-\beta/\eps-\delta}{p-1}},
\]
and so it follows that 
\begin{align*}
    |u(w)|^p\lesssim d(w,\zeta)^{p-\beta/\eps-\delta}\sum_{k\in\Z}\frac{r_k^{\beta/\eps+\delta}}{\mu(B_k)}\int_{4B_k}g_{E_Xu}^pd\mu_\beta.
\end{align*}

For $k\ge 0$, we have from the Lemma~\ref{lem:Chain} that $r_k\simeq d(x,w)$ for each $x\in 4B_k$.  Furthermore, by Theorem~\ref{thm:HypFillThm}, we have that 
\[
\mu_\beta(B_k)\simeq\mu_\beta(B(w,d(x,w)))\simeq d(x,w)^{\beta/\eps}\nu(B(w,d(x,w)))
\]
for each $x\in 4B_k$. Likewise, for $k<0$, the same comparisons hold with $w$ replaced by $\zeta$.  Therefore, letting $C_{w,\zeta}^1:=\bigcup_{k\ge 0}4B_k$ and $C_{w,\zeta}^2:=\bigcup_{k<0}4B_k$, we have that 
\begin{align*}
    |u(w)|^p&\lesssim d(w,\zeta)^{p-\beta/\eps-\delta}\left(\int_{C_{w,\zeta}^1}\frac{g_{E_Xu}(x)^p\,d(x,w)^\delta}{\nu(B(w,d(x,w)))}d\mu_\beta(x)
    +\int_{C_{w,\zeta}^2}\frac{g_{E_Xu}(x)^p\,d(x,\zeta)^\delta}{\nu(B(\zeta,d(x,\zeta)))}d\mu_\beta(x)\right).
\end{align*}
Hence, we have
\begin{align}\label{eq:I+II}
\int_{B(\zeta,r)}|u|^pd\nu&\lesssim\int_{B(\zeta,r)}\int_{C_{w,\zeta}^1}\frac{g_{E_Xu}(x)^p\,d(x,w)^\delta d(w,\zeta)^{p-\beta/\eps-\delta}}{\nu(B(w,d(x,w)))}d\mu_\beta(x)d\nu(w)\nonumber\\
    &+\int_{B(\zeta,r)}\int_{C_{w,\zeta}^2}\frac{g_{E_Xu}(x)^p\,d(x,\zeta)^\delta d(w,\zeta)^{p-\beta/\eps-\delta}}{\nu(B(\zeta,d(x,\zeta)))}d\mu_\beta(x)d\nu(w)=:I+II.
\end{align}

We first estimate $I$.  Note that there exists a constant $C\ge 1$ independent of $w$ such that $C^1_{\zeta,w}\subset B(\zeta,Cr)$.  By Tonelli's theorem, we then have that 
\begin{align*}
    I&\le\int_{B(\zeta,Cr)}g_{E_Xu}(x)^p\int_{B(\zeta,r)}\frac{d(x,w)^\delta d(w,\zeta)^{p-\beta/\eps-\delta}}{\nu(B(w,d(x,w)))}\chi_{C_{w,\zeta}^1}(x)d\nu(w)d\mu_\beta(x).
\end{align*}
Note that if $x\in C_{w,\zeta}^1$, then $d(x,w)\simeq d(x,\overline Z)$, and so there exists $C\ge 1$ so that $w\in B(x,Cd(x,\overline Z))$.  Using this, and the doubling property of $\nu$, it follows that for $x\in B(\zeta,Cr)$,
\begin{align*}
   \int_{B(\zeta,r)}\frac{d(x,w)^\delta d(w,\zeta)^{p-\beta/\eps-\delta}}{\nu(B(w,d(x,w)))}&\chi_{C_{w,\zeta}^1}(x)d\nu(w)\\
   &\lesssim r^{p-\beta/\eps}\int_{B(x,Cd(x,\overline Z))}\frac{1}{\nu(B(w,d(x,w)))}d\nu(w)\lesssim r^{p-\beta/\eps}.
\end{align*}
Hence, we have that 
\begin{equation}\label{eq:I}
I\lesssim r^{p-\beta/\eps}\int_{B(\zeta,Cr)}g_{E_Xu}^pd\mu_\beta.
\end{equation}
We estimate $II$ in a similar manner: using Tonelli's theorem and the fact that  $x\in C_{w,\zeta}^2$ implies that $d(x,\zeta)\simeq d(x,\overline Z)$, we obtain 
\[
II\lesssim r^{p-\beta/\eps}\int_{B(\zeta,Cr)}g_{E_Xu}^pd\mu_\beta.
\]

Combining this estimate with \eqref{eq:I} and \eqref{eq:I+II}, and using Theorem~\ref{thm:HypFillThm}, we obtain 
\[
\fint_{B(\zeta,r)}|u|^pd\nu\lesssim\frac{r^p}{r^{\beta/\eps}\nu(B(\zeta,r))}\int_{B(\zeta,Cr)}g_{E_Xu}^pd\mu_\beta\lesssim r^p\fint_{B(\zeta,Cr)}g_{E_Xu}^pd\mu_\beta
\]
for $\BCap_p^{\oXeps}$-q.e.\ $\zeta\in\partial Z$ and every $r>0$.  Since $p>\sigma/\theta$ by assumption, and $\beta/\eps=p(1-\theta)$, it follows that $p>\sigma+\beta/\eps$.  Therefore, by \cite[Proposition 3.11]{GKS}, the above inequality holds for $\Ha^{-(\sigma+\beta/\eps)}_{\mu_\beta}$-a.e.\ $\zeta\in\partial Z$.  Furthermore, by \cite[Lemma~3.10]{GKS}, we have that 
\[
\limsup_{r\to 0^+} r^p\fint_{B(\zeta,Cr)}g_{E_Xu}^pd\mu_\beta\le\limsup_{r\to 0^+}r^{\sigma+\beta/\eps}\fint_{B(\zeta,Cr)}g_{E_Xu}^pd\mu_\beta=0
\]
for $\Ha^{-(\sigma+\beta/\eps)}_{\mu_\beta}$-a.e.\ $\zeta\in\partial Z$.  Since $\pi\simeq\Ha^{-(\sigma+\beta/\eps)}_{\mu_\beta}|_{\overline Z}$ by Theorem~\ref{thm:HypFillThm}, it follows that 
\[
\lim_{r\to 0^+}\fint_{B(\zeta,r)}|u|^pd\nu=0.
\]
 for $\pi$-a.e.\ $\zeta\in\partial Z$.    
\end{proof}

\begin{prop}\label{prop:2}
Let $u\in B^\theta_{p,p}(\overline Z,\nu)$ be such that $T_Zu=0$\, $\BCap^{\overline Z}_{\theta,p}$-q.e.~on $\partial Z$.
Then there exist $\{u_k\}_{k\in\N}\subset B^\theta_{p,p}(\overline Z,\nu)$ with $\supt(u_k)\Subset Z$ such that $u_k\to u$ in $B^\theta_{p,p}(\overline Z,\nu)$. 
\end{prop}

\begin{proof}
    Suppose that $\zeta\in\partial Z$ is such that $T_Zu(\zeta)=0$.  Consider $E_Xu\in N^{1,p}(\overline X_\eps,\mu_\beta)$ as given by Theorem~\ref{thm:HypFillThm}, and fix $0<r<2\diam(Z)$.  Let $N_r\in\N$ be the smallest positive integer such that $\alpha^{-N_r}<r$, and for $i\ge N_r-1$, let 
    \[
    I_{i,r}:=\{j\in\N:z_{i,j}\in A_i\cap B_\eps(\zeta,r)\},
    \]
    where $A_i$ is the maximal $\alpha^{-i}$-separated set chosen in the construction of the hyperbolic filling, see Section~\ref{sec:HypFill}.  
    Recall that we adopt the convention that $B(x,r)$, $x\in\overline{X}_\eps$ and $r>0$, is the ball with respect to the metric $d_\eps$.
    However, keeping in mind that the original metric $d$ on $Z$ differs (biLipschitzly) from $d_\eps$ and that the hyperbolic filling
    used the metric $d$, we will use the notation $B_Z(\xi,r)$ for $\xi\in Z$ and $r>0$ to denote balls in $Z$ with respect to the metric $d$,
    as this is used to construct the extension operator $E_X$.
    By the construction of the extension operator $E_X$, we then have that 
    \begin{align*}
        \fint_{B(\zeta,r)}|E_Xu|d\mu_\beta&\lesssim\frac{1}{\mu_\beta(B(\zeta,r))}\sum_{i=N_r-1}^\infty\sum_{j\in I_{i,r}}\int_{B(z_{i,j},\alpha^{-i})}|E_Xu|d\mu_\beta\\
        &\lesssim\frac{1}{\mu_\beta(B(\zeta,r))}\sum_{i=N_r-1}^\infty\sum_{j\in I_{i,r}}\mu_\beta(B(z_{i,j},\alpha^{-i}))|u_{B_Z(z_{i,j},\alpha^{-i})}|\\
        &\lesssim\frac{1}{\mu_\beta(B(\zeta,r))}\sum_{i=N_r-1}^\infty\sum_{j\in I_{i,r}}\alpha^{-i\beta/\eps}\nu(B_Z(z_{i,j},\alpha^{-i}))\fint_{B_Z(z_{i,j},\alpha^{-i})}|u|d\nu\\
        &=\frac{1}{\mu_\beta(B(\zeta,r))}\sum_{i=N_r-1}^\infty\alpha^{-i\beta/\eps}\sum_{j\in I_{i,r}}\int_{B_Z(z_{i,j},\alpha^{-i})}|u|d\nu\\
        &\lesssim\frac{1}{\mu_\beta(B(\zeta,r))}\int_{B_Z(\zeta,2r)}|u|d\nu\sum_{i=N_r-1}^\infty\alpha^{-i\beta/\eps}\\
        &\lesssim\frac{r^{\beta/\eps}}{\mu_\beta(B(\zeta,r))}\int_{B_Z(\zeta,2r)}|u|d\nu\simeq \fint_{B_Z(\zeta,2r)}|u|d\nu.
    \end{align*}
    Here we have used the bounded overlap of the collection $\{B_Z(z_{i,j},\alpha^{-i})\}_{j\in I_{i,r}}$ as well as the codimensional relationship between $\mu_\beta$ and $\nu$ as given by Theorem~\ref{thm:HypFillThm}.  Since $T_Zu(\zeta)=0$, it follows that 
    \begin{equation}\label{eq:Eu Zero II }
    \lim_{r\to 0^+}\fint_{B(\zeta,r)}|E_Xu|d\mu_\beta=0.
    \end{equation}
    
    Since $T_Zu(\zeta)=0$ for 
$\BCap_{\theta,p}^{\overline Z}$-q.e.~$\zeta\in\partial Z$, it follows from~\cite[Proposition~13.2]{BBS} that~\eqref{eq:Eu Zero II } holds for
$\BCap_p^{\oXeps}$-q.e.\ $\zeta\in\partial Z$. Therefore, by \cite[Theorem~1.1]{KKST}, we have that 
$E_Xu\in N_0^{1,p}(\overline X_\eps\setminus\partial Z,\mu_\beta)$, and so by \cite[Theorem~5.46]{BB}, there exists 
$\{f_k\}_{k\in\N}\subset N^{1,p}(\overline X_\eps,\mu_\beta)$ with $\supt(f_k)\Subset\overline X_\eps\setminus\partial Z$ 
such that $f_k\to E_Xu$ in $N^{1,p}(\overline X_\eps,\mu_\beta)$.  Setting $u_k:=T_Xf_k\in B^\theta_{p,p}(\overline Z,\nu)$, 
where $T_X$ is the trace operator given by Theorem~\ref{thm:HypFillThm}, we have from the construction 
of $T_X$ that $\supt(u_k)\Subset Z$.  Furthermore, by boundedness and linearity of $T_X$, it follows that $u_k\to u$ in 
$B^\theta_{p,p}(\overline Z,\nu)$.
\end{proof}

\section{Traces of Besov spaces on $Z$, at $\partial Z$.}\label{Sec:4}

In this section we complete the set of tools we need in order to study the problem given in Definition~\ref{def:Dirichlet1}, by
constructing the trace operator $T_Z:B^\theta_{p,p}(\overline Z,\nu)\to B^{\theta-\sigma/p}_{p,p}(\partial Z,\pi)$. To do so,
in addition to the assumptions on $Z$ and $\partial Z$ adopted in the previous section, we will also assume that
the measure $\pi=\mathcal{H}_\nu^{-\sigma}$ on $\partial Z$ is $\sigma$-codimensional Ahlfors regular with respect to the measure $\nu$,
as in Definition~\ref{def:codim-regular}, and that $p>\max\{1,\sigma/\theta\}$.

\begin{lem}\label{lem:traceZ}
There is a bounded linear trace operator $T_Z:B^\theta_{p,p}(\overline Z,\nu)\to B^{\theta-\sigma/p}_{p,p}(\partial Z,\pi)$; furthermore,
we also have a bounded linear extension operator $E_Z:B^{\theta-\sigma/p}_{p,p}(\partial Z,\pi)\to B^\theta_{p,p}(\overline Z,\nu)$
such that $T_Z\circ E_Z$ is the identity map on $B^{\theta-\sigma/p}_{p,p}(\partial Z,\pi)$. 
\end{lem}

\begin{proof}
Fix $\alpha>2$, $\eps=\log\alpha$, and $\beta=\eps p(1-\theta)$ in the construction of the uniform domain $(X_\eps,d_\eps,\mu_\beta)$ for which $\overline Z=\partial X_\eps$.  Since $\overline{X}_\eps\setminus\partial Z$ is a uniform domain and $\pi$ is $(\sigma+\tfrac{\beta}{\eps})$-codimensional Ahlfors regular
with respect to the measure $\mu_\beta$, it follows from~\cite{Maly, GS} that there is an extension operator
$E_{X,\partial}:B^{\theta-\sigma/p}_{p,p}(\partial Z,\pi)\to N^{1,p}(\overline{X}_\eps,\mu_\beta)$ and a trace operator
$T_{X,\partial}:N^{1,p}(\overline{X}_\eps,\mu_\beta)\to B^{\theta-\sigma/p}_{p,p}(\partial Z,\pi)$, both bounded and linear,
with $T_{X,\partial}\circ E_{X,\partial}$ the identity map on $B^{\theta-\sigma/p}_{p,p}(\partial Z,\pi)$. 
As by construction $\nu$ is $\beta/\eps$-codimensional Ahlfors regular with respect to $\mu_\beta$, and $X_\eps$ is a 
uniform domain, we also have similar trace and extension operators
$T_{X}:N^{1,p}(\overline{X}_\eps,\mu_\beta)\to B^{\theta}_{p,p}(\overline Z,\nu)$ and 
$E_{X}:B^\theta_{p,p}(\overline Z,\nu)\to N^{1,p}(\overline{X}_\eps,\mu_\beta)$, see Theorem~\ref{thm:HypFillThm}.
Setting $E_Z=T_{X}\circ E_{X,\partial}$ yields the desired conclusion once we have constructed the trace operator $T_Z$. 
We will consider as a candidate for $T_Z$ the operator $T=T_{X,\partial}\circ E_X$. 
In order for $T$ to be a suitable trace operator, we need to verify that 
\begin{enumerate}
\item[{\bf{(a)}}] $T\circ E_Z$ is the identity map on $B^{\theta-\sigma/p}_{p,p}(\partial Z,\pi)$,
\item[{\bf{(b)}}] for each $f\in B^\theta_{p,p}(\overline Z,\nu)$, for $\pi$-a.e.~$\xi\in\partial Z$ we have that
\[
\lim_{r\to 0^+}\fint_{B(\xi,r)}|f-Tf(\xi)|\, d\nu=0.
\]
\end{enumerate}

From~\cite{BBS} we know that the class of Lipschitz functions on $Z$ and $\partial Z$ are dense in $B^\theta_{p,p}(\overline Z,\nu)$ and 
$B^{\theta-\sigma/p}_{p,p}(\partial Z,\pi)$ respectively. Lipschitz functions on $Z$ and $\partial Z$ are extended as Lipschitz
functions in $N^{1,p}(\oXeps,\mu_\beta)$ by the respective extension operators $E_X$ and $E_{X,\partial}$, and the trace operators
act on Lipschitz functions on $\overline{X}_\eps$ as restrictions, that is, if $f$ is a Lipschitz function on $X_\eps$, then 
it has a Lipschitz extension (also denoted $f$) to $\overline{X}_\eps$, and then
it can be seen from the construction of the trace operators that $T_Xf=f\vert_Z$ and $T_{X,\partial}f=f\vert_{\partial X}$.
Thus for Lipschitz functions $f\in B^\theta_{p,p}(\overline Z,\nu)$ we have that $Tf=f\vert_{\partial Z}$, and so
$T\circ E_Z(f)=f$ for Lipschitz $f\in B^{\theta-\sigma/p}_{p,p}(\partial Z,\pi)$; by the bounded linearity property of $T\circ E_Z$ it now
follows that $T\circ E_Z$ is the identity map on $B^{\theta-\sigma/p}_{p,p}(\partial Z,\pi)$ as required in~(a).

Fix $f\in B^\theta_{p,p}(\overline Z,\nu)$, and let $\xi\in\partial Z$ be a Lebesgue point of $E_Xf$ with respect to the measure $\mu_\beta$.
As noted at the end of the proof of Proposition~\ref{prop:1}, $\pi$-almost every point is such a point. This means that
\[
\lim_{r\to 0^+}\fint_{B(\xi,r)}|E_Xf-T_X\circ E_Xf(\xi)|\, d\mu_\beta=0.
\]
Moreover, we also require $\xi$ to satisfy the condition that
\begin{equation}\label{eq:Max1}
\lim_{r\to 0^+} r^p\, \fint_{B(\xi,Cr)}g_{E_Xf}^p\, d\mu_\beta=0;
\end{equation}
and we also know from \cite[Lemma~3.10]{GKS} that $\pi$-almost every $\xi\in\partial Z$ is such a point.
For the rest of this proof we will also denote $E_Xf$ by $f$ (note that $T_X(E_Xf)=f$). Fix $r>0$; then for Lebesgue points $\zeta\in B(\xi,r)\cap Z$, 
as in the proof of Proposition~\ref{prop:1} we see that with a fixed choice of $\delta$ with $0<\delta<p-\beta/\eps$,
\[
|f(\zeta)-f(\xi)|^p\lesssim d(\zeta,\xi)^{p-\beta/\eps-\delta}\, \sum_{k\in\Z}\frac{r_k^{\beta/\eps+\delta}}{\mu(B_k)}\, \int_{4B_k}g_f^p\, d\mu_\beta,
\]
with $r_k\approx d(x_k,\xi)$ when $k\ge 0$ and $x_k$ is the center of the ball $B_k$, and $r_k\approx d(x_k,\zeta)$ when $k<0$.
Proceeding further as in the proof of Proposition~\ref{prop:1}, we obtain
\[
\fint_{B(\xi,r)}|f(\zeta)-f(\xi)|^p\, d\nu(\zeta)\lesssim r^p\, \fint_{B(\xi,Cr)}g_f^p\, d\mu_\beta.
\]
Now the condition~(b) follows from~\eqref{eq:Max1}.
\end{proof}

\begin{remark}\label{rem:traces-galore}
With the above constructions in place, we have $T_{X,\partial}:N^{1,p}(\overline{X}_\eps,\mu_\beta)\to B^{\theta-\sigma/p}_{p,p}(\partial Z,\pi)$
as a bounded linear trace operator. On the other hand, we also have
$T_Z\circ T_X:N^{1,p}(\overline{X}_\eps,\mu_\beta)\to B^{\theta-\sigma/p}_{p,p}(\partial Z,\pi)$ as a bounded linear operator. For Lipschitz
functions $w\in N^{1,p}(\overline{X}_\eps,\mu_\beta)$, we know that 
$T_Z\circ T_X(w)=T_Z(w\vert_Z)$, with $w\vert_Z$ also Lipschitz continuous on $Z$; hence it has a unique Lipschitz extension to $\partial Z$,
given by $w\vert_{\partial Z}$. Thus $T_Z\circ T_X(w)=w\vert_{\partial Z}=T_{X,\partial}(w)$. Thus the two bounded linear operators agree on
the subclass of all Lipschitz functions in $N^{1,p}(\overline{X}_\eps,\mu_\beta)$. As $(\overline{X}_\eps,d_\eps,\mu_\beta)$ is doubling and 
supports a $1$-Poincar\'e inequality, it follows that Lipschitz functions form a dense subclass of $N^{1,p}(\overline{X}_\eps,\mu_\beta)$
(see for instance~\cite[Theorem~8.2.1]{HKST});
hence on $N^{1,p}(\overline{X}_\eps,\mu_\beta)$ we have that $T_Z\circ T_X=T_{X,\partial}$. 
\end{remark}

\section{Existence of solutions to the inhomogeneous Dirichlet problem}

The goal of this section is to establish the existence of solutions to the problem given in Definition~\ref{def:Dirichlet1}, and to show 
that such solutions are unique. We also prove a comparison principle for solutions, provided that the inhomogeneity data $G$ remains the same.

We continue to adopt the standing assumptions from the previous section here; in particular, $\pi=\Ha^{-\sigma}_\nu$ is $\sigma$-codimensional with respect  to $\nu$, and
$p>\max\{1,\sigma/\theta\}$.
We let $G\in L^{p'}(Z,\nu)$ and $f\in B^{\theta-\sigma/p}_{p,p}(\partial Z,\pi)$ and set the class
\[
\A(f):=\{u\in B^\theta_{p,p}(\overline Z,\nu):T_Zu=f\ \ \ \pi\text{-a.e.\ on }\partial Z\}.
\]
For each $u\in B^\theta_{p,p}(\overline Z,\nu)$, we consider the functional $I_G$ as defined in Remark~\ref{rem:Euler-Lagrange}, 
and then consider the minimization problem 
\begin{equation}\label{eq:Inhom Minimization}
\min\left\{I_G(u):u\in\A(f)\right\},
\end{equation}   
which corresponds to the inhomogeneous Dirichlet problem given in Definition~\ref{def:Dirichlet1}.

\begin{thm}\label{thm:Inhom Existence}
Let $G\in L^{p'}(Z,\nu)$ and $f\in B^{\theta-\sigma/p}_{p,p}(\partial Z,\pi)$.  Then there exists $u\in\A(f)$ such that 
\[
I_G(u)=\inf_{v\in\A(f)}I_G(v).
\]
Moreover, if $w\in\A(f)$ is another such minimizer, then $w=u$ on $Z$.
\end{thm}

\begin{remark}\label{rem:p-large-vs-unique}
Recall from our standing assumptions for this section that $p>\sigma/\theta$. Under this requirement on $p$,
we know from~\cite[Proposition~3.11]{GKS} and~\eqref{eq:CapComparison}
that any measurable set $A\subset\partial Z$ with $\pi(A)>0$ must have 
$\BCap_{\theta,p}^{\overline Z}(A)>0$. If $p$ does not satisfy this condition, then
$\BCap_{\theta,p}^{\overline Z}(\partial Z)=0$, and in this case, the above uniqueness will have
to be replaced with the statement that $w=u+c$ for some constant $c$, and in addition, we would
have to require that $\int_ZG\, d\nu=0$.
\end{remark}

\begin{proof}[Proof of Theorem~\ref{thm:Inhom Existence}]
Since $E_Zf\in\A(f)$, we have that 
\[
I:=\inf_{v\in\A(f)}I_G(v)<\infty,
\]
and so we choose $u_k\in\A(f)$ such that 
\[
I_G(E_Zf)\ge I_G(u_k)\to I,
\]
as $k\to\infty$.  By a slight abuse of notation, we will also denote by $f$ the extension $E_X\circ E_Zf\in N^{1,p}(\overline X_\eps,\mu_\beta)$ in what follows.  By the H\"older inequality and Theorem~\ref{thm:HypFillThm}, we have that 
\begin{align*}
p\int_Zu_kG\,d\nu\le p\|u_k\|_{L^p(Z,\nu)}\ &\|G\|_{L^{p'}(Z,\nu)}\lesssim\left(\|\overline{u}_k\|_{L^p(\overline X_\eps,\mu_\beta)}
 +\|\nabla{\overline u}_k\|_{L^p(\overline X_\eps,\mu_\beta)}\right)\|G\|_{L^{p'}(Z,\nu)}\\
	&\le\left(\|\overline{u}_k-f\|_{L^p(\overline X_\eps,\mu_\beta)}+\|f\|_{L^p(\overline X_\eps,\mu_\beta)}
	+\|\nabla \overline u_k\|_{L^p(\overline X_\eps,\mu_\beta)}\right)\|G\|_{L^{p'}(Z,\nu)}.
\end{align*}
Since $u_k\in\A(f)$ and since $\BCap^{\overline Z}_{\theta,p}\simeq\BCap^{\oXeps}_p$ on $\overline Z$ by \eqref{eq:CapComparison}, it follows that $\overline u_k-f=0$ $\BCap^{\oXeps}_p$-q.e.\ on $\partial Z$.  Since $\BCap_p^{\oXeps}(\partial Z)>0$,  we then have from the Maz'ya-type inequality Theorem~\ref{thm:Mazya} that
\begin{equation}\label{eq:Mazya bound}
\|\overline u_k-f\|_{L^p(\overline X_\eps,\mu_\beta)}\lesssim\|\nabla (u_k-f)\|_{L^p(\overline X_\eps,\mu_\beta)}
\le\|\nabla \overline u_k\|_{L^p(\overline X_\eps,\mu_\beta)}+\|\nabla f\|_{L^p(\overline X_\eps,\mu_\beta)}.
\end{equation}
Note that the comparison constant above depends on $p$, $\theta$, the doubling constant of $\nu$, as well as $\diam(Z)$, $\nu(Z)$, and $\BCap_{\theta,p}^{\overline Z}(\partial Z)$, but is independent of $k$. 

Substituting this into the previous expression, we obtain 
\begin{align*}
p\int_Zu_kG\,d\nu&\le C\left(\|\nabla \overline u_k\|_{L^p(\overline X_\eps,\mu_\beta)}+\|f\|_{L^p(\overline X_\eps,\mu_\beta)}
+\|\nabla f\|_{L^p(\overline X_\eps,\mu_\beta)}\right)\|G\|_{L^{p'}(Z,\nu)}\\
	&=:C_G\|\nabla \overline u_k\|_{L^p(\overline X_\eps,\mu_\beta)}+C_{f,G}.
\end{align*}
Hence, we have that 
\begin{align*}
I_G(E_Zf)&\ge I_{G}(u_k)=\|\nabla \overline u_k\|_{L^p(\overline X_\eps,\mu_\beta)}^p-p\int_{Z}u_kG\,d\nu\\
	&\ge\|\nabla \overline u_k\|_{L^p(\overline X_\eps,\mu_\beta)}^p-\left(C_G\|\nabla \overline u_k\|_{L^p(\overline X_\eps,\mu_\beta)}+C_{f,G}\right).
\end{align*}
Therefore, it follows that 
\[
\|\nabla \overline u_k\|_{L^p(\overline X_\eps,\mu_\beta)}(\|\nabla \overline u_k\|_{L^p(\overline X_\eps,\mu_\beta)}^{p-1}-C_G)\le I_G(E_Zf)+C_{f,G},
\]
where the right hand side is finite and independent of $k$.  Thus, we see that the sequence 
$\{|\nabla \overline u_k|\}_k$ is bounded in $L^p(\overline X_\eps,\mu_\beta)$, which then implies that the sequence $\{\overline u_k\}_k$ is also bounded in 
$L^p(\overline X_\eps,\mu_\beta)$ by \eqref{eq:Mazya bound}.  As such, $\{\overline u_k\}_k$ is bounded in $N^{1,p}(\overline X_\eps,\mu_\beta)$.

Since $N^{1,p}(\overline X_\eps,\mu_\beta)$ is reflexive (see for instance~\cite[Theorem~13.5.7]{HKST}), 
there exists $w\in N^{1,p}(\overline X_\eps,\mu_\beta)$ and a subsequence $\{\overline u_k\}_k$, not relabeled, such that $\overline u_k\to w$ weakly in $N^{1,p}(\overline X_\eps,\mu_\beta)$.  Therefore, by Mazur's lemma, there exists a convex combination sequence
\[
w_k:=\sum_{i=k}^{N(k)}\lambda_{k,i}\overline u_i
\]
such that $w_k\to w$ in $N^{1,p}(\overline X_\eps,\mu_\beta)$.   Passing to a subsequence if necessary, it follows from \cite[Proposition~2.3, Corollary~6.3]{BB} that 
\begin{align}\label{eq:p-energy LSC}
\int_{\overline X_\eps}|\nabla w|^p\,d\mu_\beta\le\liminf_{k\to\infty}\int_{\overline X_\eps}|\nabla w_k|^p\,d\mu_\beta.
\end{align}
Furthermore, boundedness and linearity of the trace operator $T_X$ and the H\"older inequality implies that 
\begin{align}\label{eq:trace limit}
\left|\int_Z T_Xw_kG\,d\nu-\int_Z T_XwG\,d\nu\right|&\le\|T_X(w_k-w)\|_{L^p(Z,\nu)}\|G\|_{L^{p'}(Z,\nu)}\nonumber\\
&\lesssim\|w_k-w\|_{N^{1,p}(\overline X_\eps,\mu_\beta)}\|G\|_{L^{p'}(Z,\nu)}\to 0
\end{align}
as $k\to\infty$.  By convexity of the functional $v\mapsto \int_{\overline X_\eps} |\nabla v|^pd\mu$ and linearity of $T_X$, we also have that 
\begin{align}\label{eq:w_k energy}
\int_{\overline X_\eps}|\nabla w_k|^p\,d\mu_\beta-p\int_{Z}T_Xw_k\, G\,d\nu
\le\sum_{i=k}^{N(k)}\lambda_{k,i}\left(\int_{\overline X_\eps} |\nabla \overline u_i|^p\,d\mu-p\int_Z u_iG\,d\nu\right)
=\sum_{i=k}^{N(k)}\lambda_{k,i}I_G(u_i)\to I
\end{align}
as $k\to\infty$. 

We claim that $T_Z\circ T_Xw=f$.  By the triangle inequality and linearity of the trace operators, we have that 
\begin{align*}
\|T_Z\circ T_Xw-f\|_{B^{\theta-\sigma/p}_{p,p}(\partial Z,\pi)}\le\|T_Z\circ T_X(w-w_k)\|_{B^{\theta-\sigma/p}_{p,p}(\partial Z,\pi)}+\|T_Z\circ T_Xw_k-f\|_{B^{\theta-\sigma/p}_{p,p}(\partial Z,\pi)}.
\end{align*} 
Boundedness of the trace operators gives us that 
\[
\|T_Z\circ T_X(w-w_k)\|_{B^{\theta-\sigma/p}_{p,p}(\partial Z,\pi)}\lesssim\|T_X(w-w_k)\|_{B^\theta_{p,p}(\overline Z,\nu)}
\lesssim\|w-w_k\|_{N^{1,p}(\overline X_\eps,\mu_\beta)}\to 0
\]
as $k\to\infty$.  Recall that $T_Zu_i=f$ as $u_i\in\A(f)$, and so linearity of the trace operators gives us
\[
\|T_Z\circ T_Xw_k-f\|_{B^{\theta-\sigma/p}_{p,p}(\partial Z,\pi)}
=\left\|\sum_{i=k}^{N(k)}\lambda_{k,i}T_Zu_i-f\right\|_{B^{\theta-\sigma/p}_{p,p}(\partial Z,\pi)}=0.
\]
This proves the claim, and so we have that $T_Xw\in\A(f)$.

Since $\overline{T_Xw}$ is the $p$-harmonic extension of $T_Xw$ to $\overline X_\eps$, it follows that 
\[
\int_{\overline X_\eps} |\nabla \overline{T_Xw}|^p\,d\mu_\beta\le\int_{\overline X_\eps}|\nabla w|^p\,d\mu_\beta.
\]
As $T_Xw\in\A(f)$, we then have from \eqref{eq:p-energy LSC}, \eqref{eq:trace limit}, and \eqref{eq:w_k energy} that
\begin{align*}
I\le I_G(T_Xw)=\int_{\overline X_\eps}|\nabla \overline{T_Xw}|^p\,d\mu_\beta-p\int_{Z}T_Xw&G\,d\nu\le\int_{\overline X_\eps}|\nabla w|^p\,d\mu_\beta-p\int_{Z}T_XwG\,d\nu\\
	&\le\liminf_{k\to\infty}\left(\int_{\overline X_\eps}|\nabla w_k|^p\,d\mu_\beta-p\int_{Z}T_Xw_k\,G\,d\nu\right)\le I.
\end{align*}
Therefore, setting $u:=T_Xw$, we have that $I_G(u)=I$, completing the proof of the existence.

To complete the proof of the theorem, we now suppose that $w\in\A(f)$ is another minimizer of $I_G$. Then we have
that $I_G(u)=I_G(w)$, and so
\[
\mathcal{E}_T(u,u-w)=\int_Z G(u-w)\, d\nu=\mathcal{E}_T(w,u-w).
\]
It follows that
\[
\mathcal{E}_T(u,u-w)-\mathcal{E}_T(w,u-w)=0.
\]
By the definition of $\mathcal{E}_T$, we have that
\[
\int_{X_\eps}\langle |\nabla\overline u|^{p-2}\nabla\overline u-|\nabla\overline w|^{p-2}\nabla\overline w,\ \nabla(\overline u-\overline w)\rangle\, d\mu_\beta=0.
\]
It follows that for $\mu_\beta$-a.e.~point in $\overline X_\eps$ we have $\nabla\overline u-\nabla\overline w=0$, and so
$u-w$ is constant on $X_\eps$ by the Poincar\'e inequality on $(X,d_\eps,\mu_\beta)$. As $T_Zu=T_Zw=f$, it follows that
$u=w$ $\nu$-a.e.~in $Z$.
\end{proof}

\begin{prop}[Comparison principle]\label{prop:compare}
Let $G\in L^{p'}(Z,\nu)$, and let $f_1,f_2\in B^{\theta-\sigma/p}_{p,p}(\partial Z,\pi)$ be such that $f_1\le f_2$ $\pi$-a.e.\ on $\partial Z$.  For $i=1,2$, let $u_i\in B^\theta_{p,p}(\overline Z,\nu)$ be a minimizer for \eqref{eq:Inhom Minimization} with boundary data $f_i$.  Then $u_1\le u_2$ $\nu$-a.e.\ in $Z$.  
\end{prop}

\begin{proof}
Let $u:=\min\{u_1,u_2\}$ and $v:=\max\{u_1,u_2\}$.  Since $T_Zu_1=f_1$ and $T_Zu_2=f_2$, it then follows that $T_Zu=f_1$ and $T_Zv=f_2$, that is, $u\in\A(f_1)$ and $v\in\A(f_2)$.  Let 
\[
A:=\{x\in\overline X_\eps:\overline u_1(x)>\overline u_2(x)\},
\]
where $\overline u_1,\overline u_2\in N^{1,p}(\overline X_\eps,\mu)$ are the $p$-harmonic extensions of $u_1$ and $u_2$ to $\overline X_\eps$, respectively.  Since $v\in\A(f_2)$ and $u_2$ is a minimizer of \eqref{eq:Inhom Minimization} for boundary data $f_2$, we have that 
\begin{align*}
\int_{\overline X_\eps}|\nabla \overline u_2|^pd\mu_\beta-p\int_Zu_2 G\,d\nu=I_G(u_2)\le I_G(v)
&=\int_{\overline X_\eps}|\nabla \overline v|^pd\mu_\beta-p\int_Z vG\,d\nu\\
&\le\int_{\overline X_\eps}|\nabla (\max\{\overline u_1,\overline u_2\})|^p d\mu_\beta-p\int_{Z}\max\{u_1,u_2\}G\,d\nu.
\end{align*}
Here, the last inequality follows from the fact that $\overline v$ is the $p$-harmonic extension of $v=\max\{u_1,u_2\}$ and $T_X\max\{\overline u_1,\overline u_2\}=\max\{u_1,u_2\}$.  We then have that 
\begin{align*}
\int_{\overline X_\eps}|\nabla \overline u_2|^pd\mu_\beta-p\int_Zu_2 G\,d\nu
&\le\int_A|\nabla \overline u_1|^pd\mu_\beta+\int_{\overline X_\eps\setminus A}|\nabla \overline u_2|^pd\mu_\beta
-p\int_{A\cap Z}u_1G\,d\nu-p\int_{Z\setminus A}u_2G\,d\nu,
\end{align*}
and so it follows that 
\begin{equation}\label{eq:u1>u2}
\int_A |\nabla \overline u_2|^pd\mu_\beta-p\int_{A\cap Z}u_2G\,d\nu\le\int_{A}\nabla \overline u_1|^pd\mu_\beta-p\int_{A\cap Z}u_1G\,d\nu.
\end{equation}

Using the fact that $\overline u$ is the $p$-harmonic extension of $u=\min\{u_1,u_2\}$ and that $T_X\min\{\overline u_1,\overline u_2\}=\min\{u_1,u_2\}$, along with \eqref{eq:u1>u2}, we have that 
\begin{align*}
I_G(u)&=\int_{\overline X_\eps}|\nabla \overline u|^pd\mu_\beta-p\int_{Z}uG\,\nu\\
	&\le\int_{\overline X_\eps}|\nabla(\min\{\overline u_1,\overline u_2\})|^pd\mu_\beta-p\int_Z\min\{u_1,u_2\}G\,d\nu\\
	&=\int_A |\nabla \overline u_2|^pd\mu_\beta+\int_{\overline X_\eps\setminus A}|\nabla \overline u_1|^pd\mu_\beta
	-p\int_{A\cap Z}u_2G\,d\nu-p\int_{Z\setminus A}u_1G\,d\nu\\
	&\le\int_{\overline X_\eps}|\nabla \overline u_1|^pd\mu_\beta-p\int_Zu_1G\,d\nu=I_G(u_1).
\end{align*}
As $u_1$ is a minimizer of \eqref{eq:Inhom Minimization} for boundary data $f_1$ and $u\in\A(f_1)$, it follows that $u$ is also a minimizer.  By Theorem~\ref{thm:Inhom Existence}, $u_1$ is unique, and so $\min\{u_1,u_2\}=u_1$ $\nu$-a.e.\ in $Z$, completing the proof. 
\end{proof}

\begin{remark}
A note of caution is in order here. In the case that the function $G=0$, the comparison principle also implies a maximum principle: if the boundary data $f$ satisfies $f\le M$ for some real number $M$, then the solution also is bounded above by $M$. This is because when $G=0$, the constant function $M$ is a solution with boundary data $M$, and the comparison theorem applied to these two solutions yields the maximum principle. When $G\ne 0$ however, the maximum principle need not hold, and so the comparison theorem does not yield the validity of maximum principle when the inhomogeneity data $G$ is not
the zero function. Indeed, with the boundary data $f=0$, the zero-function is \emph{never} a solution to the minimization problem when $G$ is not the zero function. 
\end{remark}

\section{Stability}

In this section, we continue to adopt the standing assumptions of the previous sections and show that 
solutions to the inhomogeneous Dirichlet problem given by Definition~\ref{def:Dirichlet1} are stable 
with respect to perturbation of both the Dirichlet data $f$ and the inhomogeneous data $G$.  
We are interested in knowing whether if the data $(f,G)$ is approximated by a sequence of data $(f_k,G_k)$, 
then the solution $u_k$ to the
$(f_k,G_k)$-inhomogeneous Dirichlet problem converges to the solution $u$ of the $(f,G)$-inhomogeneous Dirichlet 
problem. We will prove this in the final subsection below,
see Theorem~\ref{thm:stable}.
We first handle two simpler cases, namely, if the inhomogeneity data $G_k=G$ for some fixed $G$, or if the Dirichlet boundary data
$f_k=f$ for some fixed $f$; these simpler
cases are handled in the first two subsections respectively, and provide tools that are used in proving the final Theorem~\ref{thm:stable} at
the end.

We begin by proving the following gradient estimate for solutions to the $(f,G)$-inhomogeneous Dirichlet problem.  Recall that given $u\in B^\theta_{p,p}(\overline Z,\nu)$, we denote by $\overline u$ the $p$-harmonic extension of $u$ to $\overline X_\eps$.

\begin{lem}\label{lem:GradientBound}
Let $f\in B^{\theta-\sigma/p}_{p,p}(\partial Z,\pi)$ and let $G\in L^{p'}(Z,\nu)$.  If $u\in B^\theta_{p,p}(\overline Z,\nu)$ is a solution to the $(f,G)$-inhomogeneous Dirichlet problem, then 
\begin{align}\label{eq:GradientBound}
\|\nabla \overline u&\|_{L^p(\overline X_\eps,\mu_\beta)}\nonumber\\
&\le C\left(\|f\|_{B^{\theta-\sigma/p}_{p,p}(\partial Z, \pi)}^p+\|G\|_{L^{p'}(Z,\nu)}\|f\|_{B^{\theta-\sigma/p}_{p,p}(\partial Z,\pi)}+\left(1+\|G\|_{L^{p'}(Z,\nu)}\right)^{1/(p-1)}\right)=:K(f,G).
\end{align}
Here, the constant $C$ depends only on    
$p$, $\theta$, the doubling constant of $\nu$,  $\diam(Z)$, $\nu(Z)$, and $\BCap_{\theta,p}^{\overline Z}(\partial Z)$.
\end{lem}

\begin{proof}
Since we will not consider the Dirichlet data $f$ to be extended to $\oXeps$ as a $p$-harmonic extension,
for ease of notation we will denote $E_{X,\partial}f$ and $T_X\circ E_{X,\partial}f=E_Zf$ also by $f$.
Since $u$ is a solution to the $(f,G)$-inhomogeneous Dirichlet problem, we have by 
the H\"older inequality and boundedness of the trace operator $T_X$ that 
\begin{align*}
\int_{\overline X_\eps}|\nabla \overline u|^pd\mu_\beta&\le \int_{\overline X_\eps}|\nabla f|^pd\mu_\beta+p\int_Z(u-f)G\, d\nu\\
	&\le\|\nabla f\|_{L^p(\overline X_\eps,\mu_\beta)}^p+p\|G\|_{L^{p'}(Z,\nu)}\|u-f\|_{L^p(Z,\nu)}\\
	&\le\|\nabla f\|_{L^p(\overline X_\eps,\mu_\beta)}^p+C\|G\|_{L^{p'}(Z,\nu)}\left(\|\overline u-f\|_{L^p(\overline X_\eps,\mu_\beta)}+\|\nabla(\overline u-f)\|_{L^p(\overline X_\eps,\mu_\beta)}\right).
\end{align*}
Since $\overline u-f=0$ $\pi$-a.e.\ on $\partial Z$, we have from Theorem~\ref{thm:Mazya} that 
\begin{align*}
\|\overline u-f\|_{L^p(\overline X_\eps,\mu_\beta)}\le C\|\nabla(\overline u-f)\|_{L^p(\overline X_\eps,\mu_\beta)}\le C\|\nabla \overline u\|_{L^p(\overline X_\eps,\mu_\beta)}
+C\|\nabla f\|_{L^p(\overline X_\eps,\mu_\beta)}.
\end{align*}
Substituting this into the previous expression, applying the triangle inequality, and regrouping the terms yields
\begin{align*}
\|\nabla\overline u\|_{L^p(\overline X_\eps,\mu_\beta)}\left(\|\nabla\overline u\|^{p-1}_{L^p(\overline X_\eps,\mu_\beta)}-C\|G\|_{L^{p'}(Z,\nu)}\right)\le\|\nabla f\|^p_{L^p(\overline X_\eps,\mu_\beta)}+C\|G\|_{L^{p'}(Z,\nu)}\|\nabla f\|_{L^p(\overline X_\eps,\mu_\beta)}.
\end{align*}
From this, we see that either $\|\nabla\overline u\|_{L^p(\overline X_\eps,\mu_\beta)}^{p-1}\le C\|G\|_{L^{p'}(Z,\nu)}+1$ (in which case~\eqref{eq:GradientBound} holds), or 
\begin{align*}
\|\nabla\overline u\|_{L^p(\overline X_\eps,\mu_\beta)}&\le\|\nabla f\|^p_{L^p(\overline X_\eps,\mu_\beta)}+C\|G\|_{L^{p'}(Z,\nu)}\|\nabla f\|_{L^p(\overline X_\eps,\mu_\beta)}\\
	&\le C\left(\|f\|^p_{B^{\theta-\sigma/p}_{p,p}(\partial Z,\pi)}+\|G\|_{L^{p'}(Z,\nu)}\|f\|_{B^{\theta-\sigma/p}_{p,p}(\partial Z,\pi)}\right),
\end{align*}
where we have used boundedness in energy
of the extension operator $E_{X,\partial}$ in the last inequality.  Thus the conclusion of the lemma holds, with constant $C\ge 1$ depending only on  
the boundedness constants of the operators $T_X$ and $E_{X,\partial}$, and the constant in the Maz'ya inequality Theorem~\ref{thm:Mazya},
but these constants depend only on $p$, $\theta$, the doubling constant of $\nu$,  $\diam(Z)$, $\nu(Z)$, and $\BCap_{\theta,p}^{\overline Z}(\partial Z)$.
\end{proof}

\subsection{Stability with respect to the Dirichlet boundary data}\label{sub:Dirich-stab}

In this subsection we consider the stability of the Dirichlet problem when perturbing the Dirichlet boundary data. 

\begin{prop}\label{prop:GradientStability}
Let $G\in L^{p'}(Z,\nu)$, and let $f,g\in B^{\theta-\sigma/p}_{p,p}(\partial Z,\pi)$.  Let $u,v\in B^\theta_{p,p}(\overline Z,\nu)$ be solutions to the $(f,G)$- and $(g,G)$-inhomogeneous Dirichlet problems respectively, and let $K(f,G)$ and $K(g,G)$ be the respective
constants on the right hand side of \eqref{eq:GradientBound}. Then, when $p\ge 2$, we have
\begin{align}\label{eq:GradientStability 1}
\int_{\overline X_\eps}|\nabla(\overline u-\overline v)|^pd\mu_\beta\le C\left(K(f,G)^{p-1}+K(g,G)^{p-1}\right)\|f-g\|_{B^{\theta-\sigma/p}_{p,p}(\partial Z,\pi)},
\end{align}
and when $1<p<2$, we have
\begin{align}\label{eq:GradientStability 2}
\int_{\overline X_\eps}|\nabla&(\overline u-\overline v)|^pd\mu_\beta\nonumber\\
&\le C\left(K(f,G)^{p-1}+K(g,G)^{p-1}\right)^{p/2}\left(K(f,G)^p+K(g,G)^p\right)^{(2-p)/2}\|f-g\|_{B^{\theta-\sigma/p}_{p,p}(\partial Z,\pi)}^{p/2}.
\end{align}
Here the constant $C\ge 1$ depends only on $p$, $\theta$, and the doubling constant of $\nu$.   
\end{prop}

\begin{proof}
As before, we denote by $f$ and $g$ the corresponding extensions of $f$ and $g$ to 
$Z$ and $\overline X_\eps$ given by $E_Z$ and $E_{X,\partial}$.  We then have that 
$(u-f)-(v-g)=u-v-f+g\in B^\theta_{p,p}(\overline Z,\nu)$, with $T_Z(u-v-f+g)=0$ $\pi$-a.e.\ in $\partial Z$.  
From Remark~\ref{rem:Et-2nd-trace}, we know that
$\mathcal{E}_T(u,(u-f)-(v-g))$ can be computed using the $p$-harmonic extension $\overline{u}$ of $u$ to $X_\eps$ and with
the extension $(\overline{u}-f)-(\overline{v}-g)$ as the extension of $(u-f)-(v-g)$.
Since $u$ is a solution to the $(f,G)$-inhomogeneous Dirichlet problem, it follows that 
\begin{align*}
\int_{\overline X_\eps} |\nabla \overline u|^{p-2}\langle\nabla \overline u,\,\nabla(\overline u-\overline v-f+g)\rangle d\mu_\beta=\int_Z (u-v-f+g)Gd\nu,
\end{align*}
from which we obtain
\begin{align*}
\int_{\overline X_\eps} |\nabla \overline u|^{p-2}\langle\nabla \overline u,\,\nabla(\overline u-\overline v)\rangle d\mu_\beta=\int_Z (u-v-f+g)Gd\nu+\int_{\oXeps} |\nabla \overline u|^{p-2}\langle\nabla \overline u,\,\nabla(f-g)\rangle d\mu_\beta.
\end{align*}
Likewise, since $v$ is a solution to the $(g,G)$-inhomogeneous Dirichlet problem, we obtain
\begin{align*}
\int_{\oXeps} |\nabla\overline v|^{p-2}\langle \nabla\overline v,\,\nabla(\overline u-\overline v)\rangle d\mu_\beta=\int_Z (u-v-f+g)Gd\nu+\int_{\oXeps} |\nabla\overline v|^{p-2}\langle\nabla\overline v,\,\nabla(f-g)\rangle d\mu_\beta.
\end{align*}
By subtracting this from the previous equality and applying the H\"older inequality, Lemma~\ref{lem:GradientBound}, and boundedness of the extension operator $E_{X,\partial}$, we have
\begin{align}\label{eq:7-1, 1}
\int_{\oXeps}\langle|\nabla\overline u|^{p-2}\nabla\overline u-|\nabla\overline v|^{p-2}\nabla\overline v,\ &\nabla(\overline u-\overline v)\rangle d\mu_\beta=\int_{\oXeps}\langle|\nabla\overline u|^{p-2}\nabla\overline u-|\nabla\overline v|^{p-2}\nabla\overline v,\ \nabla(f-g)\rangle d\mu_\beta\nonumber\\
	&\le\int_{\oXeps}\left(|\nabla\overline u|^{p-1}+|\nabla\overline v|^{p-1}\right)|\nabla(f-g)|d\mu_\beta\nonumber\\
	&\le \left(\|\nabla\overline u\|^{p-1}_{L^p(\overline X_\eps,\mu_\beta)}+\|\nabla\overline v\|_{L^p(\overline X_\eps,\mu_\beta)}^{p-1}\right)\|\nabla(f-g)\|_{L^p(\overline X_\eps,\mu_\beta)}\nonumber\\
	&\le C\left(K(f,G)^{p-1}+K(g,G)^{p-1}\right)\|f-g\|_{B^{\theta-\sigma/p}_{p,p}(\partial Z,\pi)},
\end{align}
where $K(f,G)$ and $K(g,G)$ denote the constants on the right hand side of the  inequality in Lemma~\ref{lem:GradientBound} for $u$ and $v$ respectively. 

By~\cite[page~95]{Lind}, 
it follows that 
\begin{align}\label{eq:Ellipticity}
\langle|\nabla\overline u|^{p-2}\nabla\overline u-|\nabla\overline v|^{p-2}\nabla\overline v,\ &\nabla(\overline u-\overline v)\rangle\ge
\begin{cases}
C\, |\nabla(\overline u-\overline v)|^p,&p\ge 2,\\
C\, \left(|\nabla\overline u|+|\nabla\overline v|\right)^{p-2}|\nabla(\overline u-\overline v)|^2,& 1<p<2,
\end{cases}
\end{align}
with constant $C\ge 1$ depending only on $p$.  Thus, when $p\ge 2$, it follows from \eqref{eq:7-1, 1} that 
\begin{align*}
\int_{\oXeps}|\nabla(\overline u-\overline v)|^pd\mu_\beta&\le C\left(K(f,G)^{p-1}+K(g,G)^{p-1}\right)\|f-g\|_{B^{\theta-\sigma/p}_{p,p}(\partial Z,\pi)},
\end{align*}
which gives us \eqref{eq:GradientStability 1}.

When $1<p<2$, it follows from the H\"older inequality, \eqref{eq:Ellipticity},  \eqref{eq:7-1, 1}, and Lemma~\ref{lem:GradientBound} that 
\begin{align*}
\int_{\oXeps}&|\nabla(\overline u-\overline v)|^pd\mu_\beta\le\left(\int_{\oXeps}|\nabla(\overline u-\overline v)|^2\left(|\nabla\overline u|+|\nabla\overline v|\right)^{p-2}d\mu_\beta\right)^{p/2}\left(\int_{\oXeps}\left(|\nabla\overline u|+|\nabla\overline v|\right)^pd\mu_\beta\right)^{(2-p)/2}\\
	&\le C\left(K(f,G)^{p-1}+K(g,G)^{p-1}\right)^{p/2}\left(\|\nabla\overline u\|_{L^p(\overline X_\eps,\mu_\beta)}^p+\|\nabla\overline v\|_{L^p(\overline X_\eps,\mu_\beta)}^p\right)^{(2-p)/2}\|f-g\|_{B^{\theta-\sigma/p}_{p,p}(\partial Z,\pi)}^{p/2}\\
	&\le C\left(K(f,G)^{p-1}+K(g,G)^{p-1}\right)^{p/2}\left(K(f,G)^p+K(g,G)^p\right)^{(2-p)/2}\|f-g\|_{B^{\theta-\sigma/p}_{p,p}(\partial Z,\pi)}^{p/2},
\end{align*}
which gives \eqref{eq:GradientStability 2}.  We note that in either case, the constant $C\ge 1$ depends only on $p$, $\theta$, and the boundedness constants from the relevant trace and extension operators, which depend only on $p$, $\theta$ and the doubling constant of $\nu$.
\end{proof}

\subsection{Stability with respect to the inhomogeneity data}\label{sub:stab-inhomData}

Having established that the problem is stable under perturbation of the Dirichlet boundary data in the prior subsection, we now
turn our attention to establishing stability under perturbation of the inhomogeneity data. 

\begin{prop}\label{prop:GradientStabilityNeumann}
Let $f\in B^{\theta-\sigma/p}_{p,p}(\partial Z,\pi)$, let $G,H\in L^{p'}(Z,\nu)$, and let $u,v\in B^\theta_{p,p}(\overline Z,\nu)$ be solutions to the $(f,G)$- and $(f,H)$-inhomogeneous Dirichlet problems respectively.  Then, if $p\ge 2$, we have that 
\begin{align}\label{eq:GradientStabilityNeumann 1}
\int_{\oXeps}|\nabla(\overline u-\overline v)|^pd\mu_\beta\le C\, \|G-H\|_{L^{p'}(Z,\nu)}^{p'},
\end{align}
and when $1<p<2$, we have 
\begin{align}\label{eq:GradientStabilityNeumann 2}
\int_{\oXeps}|\nabla (\overline u-\overline v)|^pd\mu_\beta\le C\, \left(K(f,G)^p+K(f,H)^p\right)^{(2-p)}\, \|G-H\|_{L^{p'}(Z,\nu)}^p.
\end{align}
Here the constant $C\ge 1$ depends only on $p$, $\theta$, and the doubling constant of $\nu$, as well as $\diam (Z)$, $\nu(Z)$, and $\BCap_{\theta,p}^{\overline Z}(\partial Z)$. 
\end{prop}

\begin{proof}
Since $T_Zu=f=T_Zv$ $\pi$-a.e.\ on $\partial Z$, we have that $T_Z(u-v)=0$ $\pi$-a.e.\ on $\partial Z$.  Therefore, since $u$ and $v$ are solutions to the $(f,G)$- and $(f,H)$-inhomogeneous Dirichlet problems, it follows that
\begin{align*}
\int_{\oXeps}|\nabla\overline u|^{p-2}\langle\nabla\overline u,\ \nabla(\overline u-\overline v)\rangle d\mu_\beta&=\int_Z(u-v)Gd\nu
\end{align*}
as well as
\begin{align*}
\int_{\oXeps}|\nabla\overline v|^{p-2}\langle\nabla\overline v,\ \nabla(\overline u-\overline v)\rangle d\mu_\beta&=\int_Z(u-v)Hd\nu,
\end{align*}
thanks to Remark~\ref{rem:Et-2nd-trace}.

Subtracting the second equality from the first, and using H\"older's inequality, boundedness of $T_X$, and Theorem~\ref{thm:Mazya}, 
we have that
\begin{align}\label{eq:7-2, 1}
\int_{\oXeps}\langle|\nabla\overline u|^{p-2}\nabla\overline u-|\nabla\overline v|^{p-2}\nabla\overline v,\ &\nabla(\overline u-\overline v)\rangle d\mu_\beta=\int_{Z}(u-v)(G-H)d\nu\nonumber\\
	&\le\|u-v\|_{L^p(Z,\nu)}\|G-H\|_{L^{p'}(Z,\nu)}\nonumber\\
	&\le C\left(\|\nabla(\overline u-\overline v)\|_{L^p(\overline X_\eps,\mu_\beta)}+\|\overline u-\overline v\|_{L^p(\overline X_\eps,\mu_\beta)}\right)\|G-H\|_{L^{p'}(Z,\nu)}\nonumber\\
	&\le C\|\nabla(\overline u-\overline v)\|_{L^p(\overline X_\eps,\mu_\beta)}\, \|G-H\|_{L^{p'}(Z,\nu)}.
\end{align}
When $p\ge2$, it follows from \eqref{eq:Ellipticity} and \eqref{eq:7-2, 1} that 
\begin{align*}
\int_{\oXeps}|\nabla(\overline u-\overline v)|^pd\mu_\beta\le C\, \|\nabla(\overline u-\overline v)\|_{L^p(\oXeps,\mu_\beta)}\, \|G-H\|_{L^{p'}(Z,\nu)},
\end{align*}
which gives us \eqref{eq:GradientStabilityNeumann 1}.

When $1<p<2$, it follows from the H\"older inequality, \eqref{eq:Ellipticity}, \eqref{eq:7-2, 1}, and Lemma~\ref{lem:GradientBound} that
\begin{align*}
\int_{\oXeps}|\nabla(\overline u-&\overline v)|^pd\mu_\beta\le\left(\int_{\oXeps}|\nabla(\overline u-\overline v)|^2\left(|\nabla\overline u|+|\nabla\overline v|\right)^{p-2}d\mu_\beta\right)^{p/2}\left(\int_{\oXeps}\left(|\nabla\overline u|+|\nabla\overline v|\right)^pd\mu_\beta\right)^{(2-p)/2}\\
&\le C\, \left(\|\nabla(\overline u-\overline v)\|_{L^p(\overline X_\eps,\mu_\beta)}\, \|G-H\|_{L^{p'}(Z,\nu)}\right)^{p/2}\,
   \left(K(f,G)^p+K(f,H)^p\right)^{(2-p)/2},   
\end{align*}
which gives us \eqref{eq:GradientStabilityNeumann 2}.
\end{proof}

\subsection{Complete stability}\label{sub:stable}

We now have the tools needed to prove the full stability question. 

\begin{thm}\label{thm:stable}
Suppose that $(f_k)_k\subset B^{\theta-\sigma/p}_{p,p}(\partial Z,\pi)$ and $(G_k)_k\subset L^{p'}(Z,\nu)$, and
$f\in B^{\theta-\sigma/p}_{p,p}(\partial Z,\pi)$, $G\in L^{p'}(Z,\nu)$ such that $f_k\to f$ and $G_k\to G$ in their
respective function-spaces. Let $u_k$ be a solution to the $(f_k,G_k)$-inhomogeneous Dirichlet problem on $Z$ for each $k\in\N$,
and $u$ a solution to the $(f,G)$-inhomogeneous Dirichlet problem on $Z$. Then $u_k\to u$ in $B^{\theta}_{p,p}(\overline Z,\nu)$.
\end{thm}

\begin{proof}
Since $(f_k)_k$ and $(G_k)_k$ are Cauchy in their respective function-spaces, it follows that their norms are bounded.
Hence, from the definition of the functional $K$ from~\eqref{eq:GradientBound}, we have that the sequences
$\{K(f_k,G_k\}_k$ and $\{K(f_k,G)\}_k$ are bounded; set
\[
M:=\sup_k\ K(f_k,G_k)+K(f_k,G)+K(f,G).
\]

We will prove this theorem in the case $p\ge 2$, with the corresponding modifications for the case $1<p<2$ from the
two propositions yielding the corresponding results for $1<p<2$; we leave this modification for the reader.

Let $w_k$ be a solution to the $(f_k,G)$-inhomogeneous Dirichlet problem. Such a solution exists by Theorem~\ref{thm:Inhom Existence}. Then by Proposition~\ref{prop:GradientStabilityNeumann}
we have that
\begin{equation}\label{eq:stab1}
\int_{\oXeps}|\nabla(\overline u_k-\overline w_k)|^p\, d\mu_\beta
 \le C\, \|G_k-G\|_{L^{p'}(Z,\nu)}^{p'},
\end{equation}
and from Proposition~\ref{prop:GradientStability} we have that
\begin{align}\label{eq:stab2}
\int_{\oXeps}|\nabla(\overline w_k-\overline u)|^p\, d\mu_\beta
 &\le C\, \left(K(f_k,G)^{p-1}+K(f,G)^{p-1}\right)\, \|f_k-f\|_{B^{\theta-\sigma/p}_{p,p}(\partial Z,\pi)}\notag \\
   &\le 2C\, M^{p-1}\, \|f_k-f\|_{B^{\theta-\sigma/p}_{p,p}(\partial Z,\pi)}.
\end{align}
Combining~\eqref{eq:stab1} and~\eqref{eq:stab2} we see that
\[
\int_{\oXeps}|\nabla(\overline u_k-\overline u)|^p\, d\mu_\beta
 \le 2^p\, C\, (1+M^{p-1})\, \left(\|G_k-G\|_{L^{p'}(Z,\nu)}^{p'}+\|f_k-f\|_{B^{\theta-\sigma/p}_{p,p}(\partial Z,\pi)}\right),
\]
from which we conclude that $\nabla \overline u_k\to \nabla \overline u$ in $L^p(\oXeps,\mu_\beta)$. Applying Theorem~\ref{thm:Mazya} to
$(\overline u_k-f_k)-(\overline u-f)$, we see that
\begin{align*}
\|\overline u_k-\overline u\|_{L^p(\oXeps,\mu_\beta)}
&\le \|(\overline u_k-f_k)-(\overline u-f)\|_{L^p(\oXeps,\mu_\beta)}
+\|f_k-f\|_{L^p(\oXeps,\mu_\beta)}\\
&\le C\, \|\nabla [(\overline u_k-f_k)-(\overline u-f)]\|_{L^p(\oXeps,\mu_\beta)}+\|f_k-f\|_{L^p(\oXeps,\mu_\beta)}\\
&\le C\, \|\nabla (\overline u_k-\overline u)\|_{L^p(\oXeps,\mu_\beta)}
+C\, \|\nabla (f_k-f)\|_{L^p(\oXeps,\mu_\beta)}+\|f_k-f\|_{L^p(\oXeps,\mu_\beta)}\\
&\le C\, (1+M^{p-1})\,  \left(\|G_k-G\|_{L^{p'}(Z,\nu)}^{p'}+\|f_k-f\|_{B^{\theta-\sigma/p}_{p,p}(\partial Z,\pi)}
+\|f_k-f\|_{L^p(\partial Z,\pi)}\right),
\end{align*}
where we have also used the boundedness of the extension operator $E_{X,\partial}$. It follows that $\overline u_k\to \overline u$
in $N^{1,p}(\oXeps,\mu_\beta)$. Now by the continuity of the linear trace operator $T_{X,\partial}$, we see that $u_k\to u$ in
$B^\theta_{p,p}(\overline Z,\nu)$, completing the proof. 
\end{proof}

\section{Regularity}

In this section, we continue to adopt the standing assumptions of the previous sections and discuss regularity results for 
solutions to the inhomogeneous and homogeneous Dirichlet problem.  For the inhomogeneous problem, these results follow 
mutatis mutandis from earlier results, namely \cite{KS,MalShan,CKKSS}, with slight modifications.  For this reason, we 
do not include the full proofs, only descriptions of the necessary adaptations to our setting.  For the homogeneous 
problem, we relate solutions to $p$-harmonic functions in the hyperbolic filling to obtain a Kellogg-type property in Theorem~\ref{thm:HomogeneousKellogg}.  The \emph{Kellogg property} states that 
at capacity-a.e.~boundary point, if the Dirichlet boundary data is continuous, then the corresponding solution to the
Dirichlet problem has a limit to that point, from within the domain, and this limit agrees with the boundary data at that point.
For more on the Kellogg property see for instance~\cite{Lan, Bre, HedWol, Kil}, 
and for the metric setting see~\cite{BjMacSh, BBS-1,BjBjLat}.  For the nonlocal setting with $G=0$ in the Euclidean 
setting see~\cite{JaBjo, KLL}.
A criterion, called the \emph{Wiener criterion}, identifies points on the boundary
where such continuous limits exist when the problem is associated with elliptic PDEs, see for instance~\cite{Maz-1970} for the 
Euclidean setting 
and~\cite[Section~11.7]{BB} for the metric setting. Since $\overline{X}_\eps$ falls within the framework of~\cite{BB},
we have access to the Wiener criterion for the domain $\Om:=\overline{X}_\eps\setminus\partial Z$.

\subsection{Interior regularity}

In \cite{CKKSS}, the authors considered a problem closely related to the $(f,G)$-inhomogeneous Dirichlet problem
involving the operator $\mathcal{E}_T$.  However, this problem did not involve Dirichlet boundary data.  Precisely, 
given $G\in L^{p'}(Z,\nu)$ with $\int_Z G\, d\nu=0$, the problem considered in~\cite{CKKSS} was to find 
$u\in B^\theta_{p,p}(\overline Z,\nu)$
satisfying
\begin{align}\label{eq:CKKSS problem}
	\mathcal{E}_T(u,v)=&\int_{\overline{Z}} G\, v\, d\nu \text{ whenever }v\in B^\theta_{p,p}(\overline Z,\nu).
\end{align}
There, it was shown that solutions $u$ to this problem are locally bounded.  That is, if $z\in\overline Z$ and $0<R\le\diam(\overline Z)$, then 
\[
\|u\|_{L^\infty(B(z,R),\nu)}\le d,
\]
where $d>0$ is a constant depending on $\theta$, $p$, the doubling constant of $\nu$, as well as $R$, $B(z,R)$, and the quantities
\begin{equation}\label{eq:BoundQuantities}
	u(0,R):=\left(\fint_{B(z,R)}u_+^pd\mu_\beta\right)^{1/p},\quad\psi(0,R):=\fint_{B(z,R)}|G|u_+\,d\nu,
\end{equation}
see \cite[Section~8]{CKKSS}, in particular \cite[Lemmas~8.6,\,8.11]{CKKSS}.  The arguments there modify the proofs 
of~\cite[Lemma~5.10,\, Theorem~5.2]{MalShan}, where such local $L^\infty$-bounds were obtained for 
inhomogeneous data $G\in L^\infty(Z,\nu)$ by the De Giorgi method.  

Since the problem we consider in this paper involves Dirichlet boundary data $f\in B^{\theta-\sigma/p}_{p,p}(\partial Z,\pi)$ in addition to inhomogeneous data $G\in L^{p'}(Z,\nu)$, we are not able to attain local $L^\infty$-bounds on all of $\overline Z$;  we are only able to attain such bounds on balls in $Z$ which stay away from $\partial Z$.  When considering only such balls, however, the exact 
arguments from~\cite{MalShan} and~\cite{CKKSS} provide the following lemma. 

\begin{lem}\label{lem:LocalBoundedness}
	Let $f\in B^{\theta-\sigma/p}_{p,p}(\partial Z,\pi)$, let $G\in L^{p'}(Z,\nu)$, and let $u\in B^\theta_{p,p}(\overline Z,\nu)$ be a solution to the $(f,G)$-inhomogeneous Dirichlet problem.  If $z_0\in Z$ and $R_0>0$ is such that $B(z_0,2R_0)\cap\partial Z=\varnothing$, then 
	\[
	\|u\|_{L^\infty(B(z_0,R_0),\nu)}\le d,
	\]
	where $d>0$ is a constant depending only on $\theta$, $p$, the doubling constant of $\nu$ as well as $R_0$, $B(z_0,R_0)$, and the quantities $u(0,R_0)$ and $\psi(0,R_0)$ as given by \eqref{eq:BoundQuantities}.
\end{lem}   
\noindent We note that while \cite{MalShan} and \cite{CKKSS} assume that $\int_{\overline Z} G\,d\nu=0$, this 
assumption is used primarily to obtain existence and uniqueness of solutions and plays no role in the regularity results therein. 

To study H\"older regularity of solutions, we note that when $p>Q$, where $Q$ is as in \eqref{eq:LMB} with respect to the 
doubling measure $\mu_\beta$, we know that functions in $N^{1,p}(\oXeps,\mu_\beta)$, and hence functions in 
$B^\theta_{p,p}(\overline Z,\nu)$, are $(1-Q/p)$-H\"older continuous, see for example~\cite[Theorem~5.1(3)]{HaK}
or~\cite[Theorem~9.2.14]{HKST}.  As such, in the following discussion, we consider only the case $1<p\le Q$.  

In \cite[Theorem~1.6]{CKKSS}, the authors showed that solutions to the problem given by \eqref{eq:CKKSS problem} are 
locally H\"older continuous in all of $\overline Z$.  As in the case of the above local boundedness estimates, 
we are able to use the argument there verbatim to prove Theorem~\ref{thm:Holder} below, with the same H\"older exponent, provided we consider only balls in $Z$ which are far from $\partial Z$.  Again, this is due to the presence of Dirichlet data on $\partial Z$ in our problem.

If $\Omega$ is an open subset of a complete doubling metric measure space supporting a $(1,p)$-Poincar\'e inequality, 
then any $v\in N^{1,p}_\loc(\Omega)$ which is $p$-harmonic in $\Omega$ is locally $\alpha$-H\"older continuous in $\Omega$, 
where $0<\alpha\le 1$ is a constant depending only on the doubling constant and constants associated with the 
Poincar\'e inequality, see~\cite[Theorem~5.2]{KS}. 
The constant $0<\alpha\le 1$ 
appearing in Theorem~\ref{thm:Holder} below is the local H\"older regularity exponent for $p$-harmonic functions
on domains in the metric measure space
$(\overline{X}_\eps,d_\eps,\mu_\beta)$, obtained from~\cite{KS}.
As such, $\alpha$ depends only on $\theta$, $p$, the doubling constant of $\nu$, and the parameters chosen 
in the construction of $(\oXeps,d_\eps,\mu_\beta)$. 

\begin{thm}\label{thm:Holder}
	Let $\max\{1,\sigma/\theta\}<p\le Q$, where $Q$ is the lower mass bound for $\mu_\beta$, as in~\eqref{eq:LMB}.  
	Let $f\in B^{\theta-\sigma/p}_{p,p}(\partial Z,\pi)$ and let $G\in L^q(Z,\nu)\cap L^{p'}(Z,\nu)$ with 
	\[
	\frac{Q-p(1-\theta)}{p\theta}<q.
	\]
	If $u\in B^\theta_{p,p}(\overline Z,\nu)$ is a solution to the $(f,G)$-inhomogeneous Dirichlet problem, then for all $z_0\in Z$, $R_0>0$ such that $B(z_0,2R_0)\cap\partial Z=\varnothing$, and for all $z,w\in B(z_0,R_0)\cap Z$, we have that 
	\[
	|u(z)-u(w)|\le Cd(z,w)^{1-\gamma},
	\] 
	where 
	\[
	\gamma:=\max\left\{1-\alpha,\frac{p(1-\theta)(q-1)+Q}{pq}\right\}.
	\]
	That is, $u$ is locally $(1-\gamma)$-H\"older continuous in $Z$.  Here, the constant $C$ depends on $\theta$, $p$, the doubling constant of $\nu$, as well as on $R_0$, $B(z_0,R_0)$, and the constant $d$ given by Lemma~\ref{lem:LocalBoundedness}.
\end{thm}

The requirement that $G\in L^q(Z)$ is automatically satisfied if $p'>(Q-\beta/\eps)/(p-\beta/\eps)$, with the choice then of $q=p'$. 
The estimates for $q$ that guarantees the above local H\"older regularity of the solution $u$ agrees with that of~\cite[Theorem~1.2]{CSt}
when $p=2$,
where $n$, the dimension of the Euclidean space, corresponds to the lower mass bound dimension $Q-\beta/\eps$ of the measure $\nu$,
and the fractional exponent $s$ of~\cite{CSt} corresponds to $\theta=1-\beta/\eps p$.

\subsection{Kellogg-type property for the homogeneous problem}
When considering the problem given by Definition~\ref{def:Dirichlet1} with $G\equiv 0$, that is, the 
homogeneous Dirichlet problem with boundary data $f\in B^{\theta-\sigma/p}_{p,p}(\partial Z,\pi)$, we 
are able to obtain Kellogg-type results in Theorem~\ref{thm:HomogeneousKellogg} below.  
Note that here, $p$ is not required to have any relationship with $Q$; we continue to adopt the standing 
assumptions from the previous section, in particular assuming that $p>\max\{1,\sigma/\theta\}$.  
We first relate solutions to this homogeneous Dirichlet problem to solutions to the Dirichlet 
problem for $p$-harmonic functions on $\Om:=\overline{X}_\eps\setminus\partial Z$. 
Note that $\partial \Om=\partial Z$.

\begin{lem}\label{lem:p-harmonic solutions}
	Let $f\in B^{\theta-\sigma/p}_{p,p}(\partial Z,\pi)$, and denote $\Omega:=\oXeps\setminus\partial Z$.  If $u\in B^\theta_{p,p}(\overline Z,\nu)$ is a solution to the homogeneous Dirichlet problem with boundary data $f$, then the $p$-harmonic extension $\overline u$ of $u$ to $\oXeps$ is a solution to the Dirichlet problem for $p$-harmonic functions in $\Omega$ with boundary data $f$.  That is, 
	\begin{align}
		\int_\Omega|\nabla\overline u|^{p-2}\langle\nabla \overline u,\nabla v\rangle d\mu_\beta=0&\quad\text{for all }v\in N^{1,p}(\oXeps,\mu_\beta)\text{ with }T_{X,\partial}v=0\;\BCap^{\oXeps}_p\text{-q.e.\ on }\partial\Omega,\label{eq:Om-1}\\
		T_{X,\partial}\overline u=f&\quad\pi\text{-a.e.\ on }\partial\Omega. \label{eq:Om-2}
	\end{align}
	Moreover, by modifying $f$ on a $\pi$-null set if necessary, we also have that for $\BCap_p^{\overline{X}_\eps}$-q.e.~$\xi\in\partial\Om$,
	\[
	T_{X,\partial}\overline u(\xi)=f(\xi).
	\]
\end{lem}

\begin{proof}
	By our construction of the solution to the $(f,0)$-inhomogeneous Dirichlet problem and by the uniqueness of the solution, we know 
	that~\eqref{eq:Om-1} is satisfied whenever $v\in N^{1,p}(\overline X_\eps,\mu_\beta)$ with $T_{X,\partial} v=0$; here we have
	also used Remarks~\ref{rem:Et-2nd-trace} and~\ref{rem:traces-galore}. 
	We also have the validity of~\eqref{eq:Om-2} by the construction of $\overline u$,
	thanks to Remark~\ref{rem:traces-galore}.
	
	As $\overline u\in N^{1,p}(\oXeps,\mu_\beta)$ and since $\mu_\beta$ is doubling and supports a $(1,1)$-Poincar\'e inequality, it follows that $\BCap^{\oXeps}_p$-q.e.\ $\xi\in\partial Z$ is a Lebesgue point of $\overline u$, see \cite[Theorem~9.2.8]{HKST} for example.  Hence, we can modify $f$ on a set of $\pi$-measure zero so that $T_{X,\partial}\overline u=f$ $\BCap^{\oXeps}_p$-q.e.\ on $\partial Z$.\qedhere
\end{proof}

The rest of this section is devoted to the matter of boundary continuity when the Dirichlet data is continuous. For this, we assume that
$\partial Z$ is uniformly perfect (as in Definition~\ref{def:unif-perf}). 
We aim to show that every point on $\partial Z$ satisfies the Wiener criterion with respect to the domain 
$\Om=\oXeps\setminus\partial Z$, see for example \cite[Chapter~11.4]{BB}.  To do so, we first obtain the following capacity estimates: 

\begin{lem}\label{lem:vcap estimates}
	Assume that $\partial Z$ is $K$-uniformly perfect for some $K\ge 2$.  Let $z_0\in\partial Z$ and let $0<r<\diam(\partial Z)/(4K^3)$.  Then, 
	\begin{equation}\label{eq:vcap1}
		\vcapXp(B(z_0,r),\,B(z_0,2r))\lesssim\frac{\nu(B(z_0,r))}{r^{\theta p}},
	\end{equation}
	and
	\begin{equation}\label{eq:vcap2}
		\vcapXp(B(z_0,r)\cap\partial Z,\,B(z_0,2r))\gtrsim\frac{\nu(B(z_0,r))}{r^{\theta p}},
	\end{equation}
	with the comparison constants depending only on $K$, $p$, $\theta$, and the doubling constant of $\nu$.
\end{lem}

\begin{proof}
	Consider the function $u\in N^{1,p}(\oXeps,\mu_\beta)$ given by 
	\[
	u(x):=(1-(\dist(x,z_0)/r-1)_+)\chi_{B(z_0,2r)}(x),
	\]
	which satisfies $u=1$ on $B(z_0,r)$ and $u=0$ in $\oXeps\setminus B(z_0,2r)$.   Since $u$ is $1/r$-Lipschitz in 
	$\overline{X}_\eps$ with upper gradient $g_u\le \tfrac{1}{r}\, \chi_{B(z_0,2r)\setminus B(z_0,r)}$,
	by the doubling property of $\mu_\beta$ and Theorem~\ref{thm:HypFillThm} we have that 
	\begin{align*}
		\vcapXp(B(z_0,r),B(z_0,2r))\le\int_{\oXeps}\frac{1}{r^p}\chi_{B(z_0,2r)\setminus B(z_0,r)}d\mu_\beta&\lesssim\frac{\mu_\beta(B(z_0,r))}{r^p}\\
		&\simeq\frac{\nu(B(z_0,r))}{r^{p-\beta/\eps}}=\frac{\nu(B(z_0,r))}{r^{\theta p}},
	\end{align*}
	which gives us \eqref{eq:vcap1}.  Note that in the last equality, we have used the relationship $\beta/\eps=p(1-\theta)$.
	
	Let $E:=B(z_0,r)\cap\partial Z$ and let $F:=B(z_0,2K^3r)\cap\partial Z\setminus B(z_0,2r)$.  Let $\{B_i\}_{i\in\N}$ be a cover of $E$ by 
	balls that also intersect $E$, such that $\rad(B_i)\le 2K^3r$.  For each $i\in\N$, there exists a ball $\wtil B_i$ 
	centered at $E$ such that $B_i\subset\wtil B_i$ and $\rad(\wtil B_i)=2\rad(B_i)$.  By the 5-covering lemma, there 
	exists a disjoint subcollection $\{\wtil B_i\}_{i\in I\subset\N}$ such that $E\subset\bigcup_{i\in I}5\wtil B_i$.  
	Thus, by the doubling property of $\mu_\beta$ and the $(\beta/\eps+\sigma)$-codimensional relationship between 
	$\mu_\beta$ and $\pi$, we have that
	\begin{align*}
		\sum_{i\in\N}\frac{\mu_\beta(B_i)}{\rad(B_i)^{\beta/\eps+\sigma}}\gtrsim\sum_{i\in I}\frac{\mu_\beta(5\wtil B_i)}{\rad(5B_i)^{\beta/\eps+\sigma}}\simeq\sum_{i\in I}\pi(5\wtil B_i)\ge\pi(E)\simeq\frac{\mu_\beta(B(z_0,2K^3r))}{(2K^3r)^{\beta/\eps+\sigma}}.
	\end{align*}    
	Since the collection $\{B_i\}_{i\in\N}$ is arbitrary, we have that 
	\begin{equation}\label{eq:H(E)}
		\Ha^{-(\beta/\eps+\sigma)}_{\mu_\beta,2K^3r}(E)\gtrsim\frac{\mu_\beta(B(z_0,2K^3r))}{(2K^3r)^{\beta/\eps+\sigma}}.
	\end{equation}
	
	Similarly, let $\{B_i\}_{i\in\N}$ be a cover of $F$ by balls such that $\rad(B_i)\le 2K^3r$.  Choosing $\wtil B_i\supset B_i$ as above, centered at $F$ with comparable radius to $B_i$, we use the doubling property of $\mu_\beta$ and codimensionality with respect to $\pi$ to obtain 
	\[
	\sum_{i\in\N}\frac{\mu_\beta(B_i)}{\rad(B_i)^{\beta/\eps+\sigma}}\gtrsim\sum_{i\in I}\frac{\mu_\beta(5\wtil B_i)}{\rad(B_i)^{\beta/\eps+\sigma}}\simeq\sum_{i\in I}\pi(5\wtil B_i)\ge\pi(F).
	\]
	Since $\partial Z$ is $K$-uniformly perfect, and since $\partial Z\setminus B(z_0,2K^3r)\ne\varnothing$, there exists $\zeta\in \partial Z\cap B(z_0,2K^2r)\setminus B(z_0,2Kr)$, and so it follows that $B(\zeta,2(K-1)r)\cap\partial Z\subset F$.  Using the codimensional relationship of $\mu_\beta$ and $\pi$ along with the doubling property of $\mu_\beta$ then gives us that 
	\[
	\sum_{i\in\N}\frac{\mu_\beta(B_i)}{\rad(B_i)^{\beta/\eps+\sigma}}\gtrsim\pi(B(\zeta,2(K-1)r)\cap\partial Z)\simeq\frac{\mu_\beta(B(\zeta,2(K-1)r))}{(2(K-1)r)^{\beta/\eps+\sigma}}\simeq\frac{\mu_\beta(B(z_0,2K^3r))}{(2K^3r)^{\beta/\eps+\sigma}}.
	\]
	Since $\{B_i\}_{i\in\N}$ is arbitrary, we then have that 
	\begin{equation}\label{eq:H(F)}
		\Ha^{-(\beta/\eps+\sigma)}_{\mu_\beta,2K^3r}(F)\gtrsim\frac{\mu_\beta(B(z_0,2K^3r))}{(2K^3r)^{\beta/\eps+\sigma}}.
	\end{equation}
	
	Let $u\in N^{1,p}(\oXeps,\mu_\beta)$ be such that $u=1$ on $B(z_0,r)\cap\partial Z=E$ and $u=0$ in $\oXeps\setminus B(z_0,2r)\supset F$.  Note that $\beta/\eps+\sigma<p$ by our assumption that $p>\sigma/\theta$, and since $\oXeps$ supports a $(1,1)$-Poincar\'e inequality, it supports a $(1,q)$-Poincar\'e inequality for all $1<q<\infty$.  Therefore, by \eqref{eq:H(E)} and \eqref{eq:H(F)}, 
	we apply \cite[Lemma~3.1]{EGKS} to obtain
	\begin{align*}
		\int_{\oXeps}g_u^p\,d\mu_\beta\ge\int_{B(z_0,2K^3r)}g_u^p\,d\mu_\beta\gtrsim\frac{\mu_\beta(B(z_0,r))}{r^p}\simeq\frac{\nu(B(z_0,r))}{r^{p-\beta/\eps}}=\frac{\nu(B(z_0,r))}{r^{\theta p}},
	\end{align*}
	where we have again used the doubling property of $\mu_\beta$ as well as the codimensional relationship between $\mu_\beta$ and $\nu$.   
	This inequality gives us \eqref{eq:vcap2}, as $u$ is an arbitrary admissible function for computing $\vcapXp(B(z_0,r)\cap\partial Z,\,B(z_0,2r))$.
\end{proof}

Using the previous lemmas, we now obtain the following boundary regularity result for the homogeneous Dirichlet problem.

\begin{thm}\label{thm:HomogeneousKellogg}
	Suppose that $\partial Z$ is uniformly perfect, and let $f\in C(\partial Z)\cap B^{\theta-\sigma/p}_{p,p}(\partial Z,\pi)$.  If $u\in B^\theta_{p,p}(\overline Z,\nu)$ is a solution to the homogeneous Dirichlet problem with boundary data $f$, then for all $z_0\in\partial Z$, we have that 
	\[
	\lim_{Z\ni z\to z_0}u(z)=f(z_0).
	\]
\end{thm}

\begin{proof}
	By Lemma~\ref{lem:vcap estimates}, it follows that each $z_0\in\partial Z$ satisfies the Wiener criterion with respect to the domain $\Om=\oXeps\setminus\partial Z$:
	\begin{align*}
		\int_{0}^{\delta}\left(\frac{\vcapXp(B(z_0,r)\cap\partial Z,\, B(z_0,2r))}{\vcapXp(B(z_0,r),B(z_0,2r))}\right)^{1/(p-1)}\frac{dr}{r}\gtrsim\int_0^\delta\frac{1}{r}dr=\infty,
	\end{align*}
	where $\delta=\diam(\partial Z)/(4K^3)$ with $K\ge 2$ the uniform perfectness constant of $\partial Z$.  Therefore, 
	by~\cite[Theorem~11.24]{BB} or~\cite{BjMacSh} for example, it follows that for each $z_0\in\partial Z$, the following holds: for 
	every $f\in C(\partial Z)\cap B^{\theta-\sigma/p}_{p,p}(\partial Z,\pi)$, if $v$ is $p$-harmonic in 
	$\oXeps\setminus\partial Z$ with $T_{X,\partial} v=f$ $\BCap_p^{\overline{X}_\eps}$-a.e.~on $\partial Z$,  then 
	\[
	\lim_{\oXeps\setminus\partial Z\ni x\to z_0}v(x)=f(z_0).
	\]
	For such $f$, we know from Lemma~\ref{lem:p-harmonic solutions} that if $u\in B^\theta_{p,p}(\overline Z,\nu)$ is 
	a solution to the homogeneous Dirichlet problem with boundary data $f$, then $\overline u$ is $p$-harmonic 
	in $\oXeps\setminus\partial Z$ and $T_{X,\partial}\overline u=f$ $\BCap_p^{\overline{X}_\eps}$-a.e.~in $\partial Z$. 
	As $\overline{u}$ is continuous on $\Om=\overline{X}_\eps\setminus\partial Z$ by~\cite{KS}, it follows
	from the continuity of $f$ that $\overline u$ is continuous on $\overline{X}_\eps$.
	Hence, since $u=T_X\overline u$ on $Z$, it follows that 
	\[
	\lim_{Z\ni z\to z_0}u(z)=f(z_0).\qedhere
	\]  
\end{proof}

\vskip .5cm

\noindent{\bf Supporting data:} The research outcomes described in this paper does not depend on supporting data nor does it generate
supporing data sets.

%
%
%
%
%
%
%
%
%

\end{document}